\documentclass[12pt, twosided]{amsart}
\usepackage{eucal}
\usepackage{amsthm,amssymb,amscd,amsmath} 
\usepackage{color}
\usepackage[colorlinks,unicode]{hyperref}
\usepackage{enumitem}
 \usepackage[titletoc,toc,title]{appendix}
\setlist[enumerate,1]{label=\textup{(\alph*)}}

\usepackage[backgroundcolor = lightgray,
            linecolor = lightgray,
            textsize = tiny]{todonotes}

\usepackage{pagecolor}
\definecolor{solarized-base03}{HTML}{002b36}
\definecolor{solarized-base02}{HTML}{073642}
\definecolor{solarized-base01}{HTML}{586e75}
\definecolor{solarized-base00}{HTML}{657b83}
\definecolor{solarized-base0}{HTML}{839496}
\definecolor{solarized-base1}{HTML}{93A1A1}
\definecolor{solarized-base2}{HTML}{EEE8D5}
\definecolor{solarized-base3}{HTML}{FDF6E3}
\definecolor{solarized-yellow}{HTML}{B58900}
\definecolor{solarized-orange}{HTML}{CB4B16}
\definecolor{solarized-red}{HTML}{DC322F}
\definecolor{solarized-magenta}{HTML}{D33682}
\definecolor{solarized-violet}{HTML}{6C71C4}
\definecolor{solarized-blue}{HTML}{268BD2}
\definecolor{solarized-cyan}{HTML}{2AA198}
\definecolor{solarized-green}{HTML}{859900}

\numberwithin{equation}{section}
\newtheorem{theorem}{Theorem}[section]
\newtheorem{remark}[theorem]{Remark}
\newtheorem*{iremark}{Remark}
\newtheorem*{noremark}{Notational Remark} %

\newtheorem*{mthm}{Main Theorem}
\newtheorem*{mresult}{Main Result}
\newtheorem{proposition}[theorem]{Proposition}

\newtheorem{corollary}[theorem]{Corollary}
\newtheorem*{ocorollary}{Corollary}
\newtheorem{lemma}[theorem]{Lemma}

\newtheorem*{isp}{Inverse spectral problem}
\newtheorem*{idp}{Inverse dynamical problem}
\newtheorem{definition}[theorem]{Definition}
\newcommand{\N}{\mathbb{N}}
\newcommand{\Z}{\mathbb{Z}}
\newcommand{\R}{\mathbb{R}}

\newcommand{\D}{\mathbb{D}}
\newcommand{\T}{\mathbb{T}}
\newcommand{\A}{\mathbb{A}}

\newcommand{\cD}{\mathcal{D}}
\newcommand{\Om}{\Omega}

\newcommand{\inv}{^{-1}}
\newcommand{\id}{\textup{Id}}

\newcommand{\cM}{\mathcal M}
\newcommand{\invo}{\mathcal I}
\newcommand{\refl}{\mathcal R}
\newcommand{\Ls}{L^1_{*,\symm}}

\newcommand{\Line}{L}
\newcommand{\eex}[1]{#1^{\dagger}}
\newcommand{\xA}{\eex A}
\newcommand{\xhA}{\eex {\hat A}}

\newcommand{\xF}{\eex F}

\newcommand{\Coc}{\Psi} 
\newcommand{\rema}{\mathcal R}  
\newcommand{\I}{\text{I}}
\newcommand{\II}{\text{II}}
\newcommand{\admX}{X_{*,\expo}}

\newcommand{\xll}{x_\Lazf}\newcommand{\yll}{y_\Lazf}

\newcommand{\ifrac}[2]{#1/#2}

\newcommand{\family}[1]{(#1_\famt)_{|\famt| \le 1}}
\newcommand{\famt}{\tau}
\newcommand{\famgammap}{\gamma}
\newcommand{\famgammam}[1]{\famgammap(#1,\cdot)}
\newcommand{\famgamma}[2]{\famgammap(#1,#2)}
\newcommand{\dotProd}[2]{\langle#1,#2\rangle}
\newcommand{\Norm}{N}
\newcommand{\dist}{\textup{dist}}
\newcommand\resp{resp.\ }
\newcommand\parball{[-1,1]}
\newcommand{\symm}{\textup{sym}}
\newcommand{\binf}{b} 
\newcommand{\nl}{u} 
\newcommand{\nlf}{\hat\nl} 
\newcommand{\expo}{\gamma}
\newcommand{\LS}{\mathcal L} 
\newcommand{\Sp}{\text{Sp}}
\newcommand{\perim}[1]{l_{\partial #1}}

\newcommand{\lop}{\mathcal{T}}
\newcommand{\tlop}{\tilde\lop}
\newcommand{\tlopr}{\tilde\lop_{*,\text{R}}}

\newcommand{\nc}[1]{_{C^{#1}}}

\newcommand{\rightSmoothness}{8}

\newcommand{\st}{\ \textup{s.t.}\ }
\newcommand{\nsob}{_{\expo}}

\newcommand{\defspace}{C_\symm^r}

\newcommand{\ie}{i.e.\ }
\newcommand{\eg}{e.g.\ }

\newcommand{\Lq}{L_q}
\newcommand{\lPrivate}{\ell}
\newcommand{\lfunc}[1]{\lPrivate_{#1}}
\newcommand{\tlfunc}[1]{\tilde\lPrivate_{#1}}

\newcommand{\Lazf}{\El}
\newcommand{\lazf}{\el}

\newcommand{\cO}{O}
\newcommand{\dec}[2]{\delta_{#1|#2}}

\def\eps{\varepsilon}
\def\vphi{\varphi}

\def\A{\mathbb A}
\def\R{\mathbb R}
\def\T{\mathbb T}

\def\Z{\mathbb Z}
\def\S{\mathbb T^1}

\def\cM{\mathcal M}

\def\cD{\mathcal D}
\def\cS{\mathcal S}

\def\~{\tilde}

\def\dt{\delta}
\def\lb{\lambda}

\def\p{\partial}

\def\bdef{\begin{definition}}
\def\endef{\end{definition}}
\def\bthm{\begin{theorem}}
\def\ethm{\end{theorem}}
\def\blm{\begin{lemma}}
\def\elm{\end{lemma}}
\def\brm{\begin{remark}}
\def\erm{\end{remark}}
\def\bprop{\begin{proposition}}
\def\eprop{\end{proposition}}
\def\bcor{\begin{corollary}}
\def\ecor{\end{corollary}}
\def\be{\begin{eqnarray}}
\def\ee{\end{eqnarray}}
\def\beal{\begin{aligned}}
\def\enal{\end{aligned}}
\def\beaa{\begin{eqnarray*}}  \def\eeaa{\end{eqnarray*}}
\def\bea{\begin{eqnarray}}            \def\eea{\end{eqnarray}}

\newcommand{\svect}{{\bf s}}
\newcommand{\apar}{\xi} 

\newcommand{\Sop}{\Sigma}
\newcommand{\Sfo}[2]{\sigma_{#2,#1}}
\DeclareFontFamily{U}{wncy}{}
\DeclareFontShape{U}{wncy}{m}{n}{<->wncyr10}{}
\DeclareSymbolFont{mcy}{U}{wncy}{m}{n}
\DeclareMathSymbol{\El}{\mathord}{mcy}{"4C}
\DeclareMathSymbol{\el}{\mathord}{mcy}{"6C}

\begin{document}
\title[Dynamical Spectral Rigidity]
{Dynamical Spectral rigidity among\\
  \texorpdfstring{$\Z_2$}{ℤ₂}-symmetric strictly convex domains\\
  close to a circle}%
\hypersetup{pdftitle={Dynamical Spectral rigidity among ℤ₂-symmetric strictly convex domains close to a circle},
    pdfauthor={Jacopo De Simoi, Vadim Kaloshin, Qiaoling Wei, with an appendix by Hamid Hezari and the authors}
}
\author{Jacopo De Simoi}
\address{Jacopo De Simoi\\
  Department of Mathematics\\
  University of Toronto\\
  40 St George St. Toronto, ON, Canada M5S 2E4} \email{{\tt jacopods@math.utoronto.ca}}
\urladdr{\href{http://www.math.utoronto.ca/jacopods}{http://www.math.utoronto.ca/jacopods}} %

\address{Hamid Hezari\\
Department of Mathematics\\
University of California, Irvine\\
Irvine, CA 92697, USA.} \email{{\tt hezari@math.uci.edu}}
\urladdr{\href{http://www.math.uci.edu/~hezari/}{http://www.math.uci.edu/~hezari/}}%

\author{Vadim Kaloshin}
\address{Vadim Kaloshin\\
  Department of Mathematics\\
  University of Maryland, College Park\\
  20742 College Park, MD, USA.}  \email{{\tt vadim.kaloshin@gmail.com}}
\urladdr{\href{http://www.math.umd.edu/~vkaloshi/}{http://www.math.umd.edu/~vkaloshi/}}%
\author{Qiaoling Wei}
\address{Qiaoling Wei\\
  Department of Mathematics, Capital Normal University, Beijing 100048, PR CHINA}  \email{{\tt wql03ster@gmail.com}}
\begin{abstract}
  We show that any sufficiently (finitely) smooth $\Z_2$-symmetric strictly convex domain sufficiently close to a circle
  is dynamically spectrally rigid, \ie all deformations among domains in the same class which preserve the length of all
  periodic orbits of the associated billiard flow must necessarily be isometric deformations.  This gives a partial
  answer to a question of P. Sarnak (see ~\cite{Sa}).
\end{abstract}

\maketitle

\section{Introduction}
In this paper we study a problem motivated by the famous question of M. Kac~\cite{Ka}: ``Can one hear the shape of a
drum?''  More formally: let $\Om\subset \R^2$ be a planar domain, and denote by
$\Sp(\Omega) = \{0 < \lambda_0\le\lambda_1\le\cdots\le\lambda_k\le\cdots\}$ the \emph{Laplace Spectrum} of $\Om$ with
some specified boundary conditions (\eg one can consider Dirichlet boundary conditions\footnote{ From the physical point
  of view, the Dirichlet eigenvalues $\lb$ correspond to the eigenfrequencies of a vibrating membrane of shape $\Om$
  which is fixed along its boundary.}).  In other words, $\Sp(\Om)$ is the set (with multiplicities) of positive real
numbers $\lb$ that satisfy the eigenvalue problem
\begin{align*}
  \Delta u + \lb^2 u&=0, &
  u &= 0 \text{ on }\partial \Om.
\end{align*}
Given a class $\cM$ of domains and a domain $\Om\in\cM$, we say that $\Om$ is \emph{spectrally determined in $\cM$} if
it is the unique element (modulo isometries) of $\cM$ with its Laplace Spectrum: if $\Om,\Om'\in\cM$ are {\it isospetral}, i.e. $\Sp(\Om') = \Sp(\Om)$, then $\Om'$ is the image of $\Om$ by an isometry (\ie a composition of translations and
rotations).

The question of Kac can be thus formulated as follows, assuming we have fixed a class of domains $\cM$:
\begin{isp}
  Is every $\Om\in\cM$ spectrally determined?
\end{isp}
If $\cM$ is the space of all planar domains, the answer is well known to be negative (see \eg \cite{GWW}, which
generalizes some results previously obtained for compact manifolds without boundary (see~\cite{Sunada,Vi})).\footnote{
  Remarkably, Sunada (see~\cite{Sunada}) exhibits isospectral sets (\ie sets of isospectral manifolds) of arbitrarily
  large cardinality.}  However, all known examples of domains that are not spectrally determined are not convex,
moreover, they are bounded by curves that are only piecewise analytic (e.g. plane domains with corners).  On the other hand, Zelditch proved
in~\cite{ZelditchAnalytic} that the inverse spectral problem has a positive answer when $\cM$ is a generic class of
analytic $\Z_2$-symmetric convex domains (\ie symmetric with respect to reflection about a given axis).

The problem for non-analytic domains is substantially more challenging.  In the $C^\infty$ category,
Osgood-Phillips-Sarnak~\cite{OPS1,OPS2,OPS3} showed that isospectral sets are necessarily compact in the $C^\infty$
topology.  Sarnak (see~\cite{Sa}) also conjectured that an isospectral set consists of isolated domains.  In other
words, $C^\infty$-close to a $C^\infty$ domain there should be no isospectral domains, except those that can be obtained
by an isometry.  A weaker version of this conjecture can be stated as follows: a domain $\Om$ is said to be
\emph{spectrally rigid in $\cM$} if any $C^1$-smooth one-parameter isospectral family $\family{\Om}\subset\cM$ with
$\Om_0 = \Om$ is necessarily an isometric family.  We can then ask: ``Are all $C^\infty$ domains spectrally rigid?''

The problem of spectral rigidity is in principle much simpler than the inverse spectral problem;
yet it turns out to be extremely challenging.  Hezari--Zelditch (see~\cite{HZ}) provided a result
in the affirmative direction: let $\Om_0$ be bounded by an ellipse $\mathcal E$, then any
one-parameter isospectral $C^\infty$-deformation $\family\Om$ which additionally preserves
the $\Z_2 \times \Z_2$ symmetry group of the ellipse is necessarily flat (i.e., all derivatives have
to vanish for $\famt = 0$).\footnote{ Results of this kind are usually referred to as
\emph{infinitesimal spectral rigidity}.}  Popov--Topalov~\cite{PT} recently extended these results
(see also~\cite{PT1}).

Further historical remarks on the inverse spectral problem can also be found in~\cite{HZ} and
in the surveys~\cite{ZelditchSurvey04} and~\cite{ZelditchSurvey15}.

\subsection{The Length Spectrum and its relation with the Laplace spectrum}
There is a remarkable relation between the Laplace spectrum of a domain and a dynamically defined object
that we now proceed to define.  The \emph{Length Spectrum of $\Om$} is defined as the set
\begin{align*}
  \LS(\Omega) = \N\,\{\text{length of all closed geodesics of }\Omega\} \cup \N\,\{\perim{\Omega}\},
\end{align*}
where $\perim\Omega$ denotes the length of the boundary $\partial\Omega$ and $\N = \{1,2,\cdots\}$.  By \emph{closed
  geodesic of $\Om$} above we mean a periodic trajectory of the billiard flow (\ie geodesic flow in the interior of
$\Om$ with optical reflections on $\partial\Om$).

\vskip 0.1in

Andersson--Melrose (see~\cite[Theorem (0.5)]{AM}, which generalized earlier results in~\cite{Chaza,Dui}) showed that, for
strictly convex $C^{\infty}$ domains, the following relation between the singular support of the \emph{wave trace} and
the Length Spectrum holds:
\begin{align} \label{eq:laplace-length}
  \textup{sing\ supp}\left(t\mapsto \sum_{\lb_j\in  \Sp(\Om)}
  \exp (i \lb_j\,t)\right)\subset\pm\mathcal
  L(\Om)\cup\{0\}.
\end{align}
Indeed, the above inclusion holds for non-convex $C^\infty$ domains in arbitrary dimension (see~\cite[Theorem
5.4.6]{PS}).  Moreover, under generic conditions (see Remark~\ref{rm:spectra-equality} for more details) it can be shown
that the above inclusion is indeed an equality and the Laplace Spectrum determines the Length Spectrum.

It is natural to pose the same questions as above in this dynamical setting.  We say that $\Om$ is \emph{dynamically
  spectrally determined in $\cM$} if it is the unique element (modulo isometries) of $\cM$ with its Length Spectrum.

\begin{idp}
  Is every $\Om\in\cM$ dynamically spectrally determined?
\end{idp}

All counterexamples to the inverse spectral problem mentioned earlier also constitute counterexamples to the inverse
dynamical problem.  Likewise, at present, there is no known counterexample realized by convex domains.  Moreover, the
above mentioned result by Zelditch (in~\cite{ZelditchAnalytic}) also holds in the dynamical context.  In the case of
sufficiently smooth convex domain, the problem is open and presents the same challenges as the inverse spectral problem.
Let us now define the dynamical notion corresponding to spectral rigidity: we say that a domain $\Om_0\in\cM$ is
\emph{dynamically spectrally rigid in $\cM$} if any $C^1$-smooth one-parameter dynamically isospectral family
$\family\Om\subset\cM$ is necessarily an isometric family.  We can now present our result, which will be more precisely
stated in Section~\ref{s_statements}.

\vskip 0.1in

\begin{mresult}
  Let $\cM$ be the set of strictly convex domains with sufficiently (finitely) smooth boundary,
  axial symmetry and that are sufficiently close to a circle.  Then any $\Om\in\cM$
  is dynamically  spectrally rigid in $\cM$.
\end{mresult}

\subsection{Related prior results}
The problem of isospectral deformations of manifolds without boundary were considered
in some early works on variations of the spectral functions and wave invariants.

\vskip 0.1in

Let $(M,g)$ be a compact boundaryless Riemannian manifold.  A family
$(g_\famt)_{|\famt| \le 1}$ of Riemannian metrics on $M$
depending smoothly on the parameter $|\famt|\le 1$ is called a \emph{deformation
of the metric $g$} if $g_0 = g$.  A deformation is called
{\it trivial} if there exists a one-parameter
family of diffeomorphisms $\vphi_\famt:M \to M$ such that $\vphi_0 = \id$, and
$g_\famt = (\vphi_\famt)^*\,g_0$. For each homotopy class
of closed curves in $M$,
consider the infimum of $g$-lengths of curves belonging to
the given homotopy class.
The \emph{Length Spectrum} $\mathcal L(M,g)$ is defined as the union of these lengths over
all homotopy classes.  The \emph{inverse spectral problem} in this setting is to show that
two metrics with the same Length Spectrum are isometric.

\vskip 0.1in

Likewise, a deformation $(g_\famt)_{|\famt| \le 1}$ is said to
be {\it isospectral} if
$\mathcal L(M,g_\famt)=\mathcal L(M,g)$.  We say that a Riemannian manifold $(M,g)$ is
{\it spectrally rigid} if it does not admit non-trivial isospectral deformations.

Guillemin--Kazhdan in~\cite{GKaz} showed that any negatively curved surface is spectrally
rigid among negatively curved surfaces.  This result has been later extended to compact
manifolds of negative curvature in~\cite{CS}. Remark that an open question is that if one can generalize the result of ~\cite{GKaz} to hyperbolic billiards.

\vskip 0.1in

Our result is {\it an analog of~\cite{GKaz} 
  for $\Z_2$-{}symmetric convex domains close to a disk.}

\vskip 0.1in

It is also worth mentioning that for such systems there is a partial solution of the inverse spectral problem due
independently to Croke \cite{Cr} and Otal~\cite{Ot} which can be stated as follows: any negatively curved manifold is
uniquely determined by its \emph{Marked Length Spectrum}.\footnote{ The \emph{Marked Length Spectrum} is given by the
  collection of lengths of closed geodesics paired with their homotopy type.}

\vskip 0.1in

Another example of deformational spectral rigidity appears in De la Llave, Marco and Moriy\'on~\cite{dMM}.  Recall that
one can associate to a symplectic map a generating function. Then, for each periodic orbit, one can define the
corresponding \emph{action} by summing the generating function along the orbit.  This value of the action is invariant
under symplectic coordinate changes.  The union of the values all these actions over all periodic orbits is called the
action spectrum of the symplectic map.  In~\cite{dMM}, it is shown that there are no non-trivial deformations of
exact symplectic mappings $B_\famt,\ \famt \in [-1,1]$, leaving the action spectrum fixed, when $B_\famt$ are
Anosov's mappings on a symplectic manifold.  One of the reasons for symplectic rigidity in~\cite{dMM} is that all
periodic points of $B_\famt$ are hyperbolic and form a dense set.

\subsection*{Outline of our paper}
In Section~\ref{s_statements}, after introducing the necessary objects, we give a formal statement of the main
result.  In Section~\ref{s_normalized}, we reduce any family $\family\Om$ of axially symmetric domains to a
{\it normalized} family $\family{\tilde\Om}$ by rotations and translations, so that they share the same
symmetry axis and their boundaries share a common point on this axis; we then restate our main result for
normalized families (see Theorem~\ref{t_main''}).  In Section~\ref{s_billiard}, we prove the existence of
maximal symmetric periodic orbits of period $q$ and rotation number $1/q$ for any $q>1$, \ie axially symmetric
$q$-gons of maximal perimeter.  If a family $\family\Om$ is isospectral, then to each orbit of this type we
can associate an isoperimetric functional $\ell_{\Om,q}$ which vanishes in the direction of the perturbation.
Using these linear functionals $(\ell_{\Om,q})_{q>1}$ we define {\it a linearized isospectral operator}
$\lop_\Om$ (see~\eqref{eq:LOO-def}) and reduce the main result to the claim that this operator $\lop_\Om$ is
injective (Theorem~\ref{t_main}).  In Section~\ref{s_proof}, we introduce a modification of Lazutkin
coordinates designed to study the behaviour of $\lop_\Om$ and we prove Theorem~\ref{t_main} using the modified
Lazutkin coordinates and some explicit computations that we obtain in Appendix~\ref{Appendix-with-Hezari},
which is joint with H. Hezari.  In Section~\ref{s_finiteDim}, we outline the proof of a generalization of our
result to domains that are not necessarily close to a circle.  In Section~\ref{s_conclusions}, we add some
remarks on the challenges that we expect when trying to prove our result in a more general setting.  In
Appendix~\ref{s_Lazutkin}, we derive required properties of the modified Lazutkin coordinates.

\section{Definitions and statement of results}\label{s_statements}
We now provide more precise definitions of the objects introduced in the previous introductory section in the billiard table setting.

Denote by $\cD^r$ the set of strictly convex open planar domains $\Om$ whose boundary is $C^{r+1}$ smooth\footnote{ We
  use the superscript with $r$, because the associated billiard map is $C^r$ (see Section~\ref{s_billiard}).} . {For
  each domain $\Omega\in \cD^r$ denote by $\rho_\Omega$ the radius of the curvature of the boundary $\partial \Omega$.}
We will always consider the underlying class of domains $\cM\subset\cD^r$ for\footnote{ The fact that the boundary is at
  least $C^3$ guarantees that the (broken) geodesic flow is complete (see e.g.~\cite{Halpern}).}  $r\ge2$.  By
convention, we set the \emph{positive} orientation of $\p\Om$ to be counterclockwise.

By definition, geodesics in a bounded planar domain Ω are geodesics (straight lines) that get reflected at the boundary
according to the optical law ``angle of reflection = angle of incidence''. Such geodesics are often called broken geodesics. In particular,
\begin{definition}
A \emph{closed geodesic} in $\Om$ is a (not necessarily convex) polygon
inscribed in $\p\Om$ such that at each vertex, the angles formed by each of
the two sides joining at the vertex with the tangent line to $\p\Om$ are equal.
The  \emph{perimeter} of the polygon is called the \emph{length of the geodesic}.
\end{definition}
\begin{definition}
  The \emph{Length Spectrum}
  of a domain $\Omega$ is the set of positive real numbers
\begin{align*}
  \LS(\Omega) = \N\,\{\text{length of all closed geodesics of }
  \Omega\} \cup \N\,\{\perim{\Omega}\},
\end{align*}
where $\perim\Omega$ is the length of the boundary $\partial\Omega$.
\end{definition}
Let us introduce the notion of a deformation of a domain.  Recall the standard notation $\S = \R/\Z$.
\begin{definition}
  We say that $\family\Om$ is a \emph{$C^1$ one-parameter family of domains in $\cM$} if $\Om_\famt\in\cM$ for any
  $|\famt|\le1$ and there exists
  \begin{align*}
    \famgammap:\parball\times\S\to\R^2
  \end{align*}
  such that $\famgamma \famt{\apar}$ is continuously differentiable in $\famt$ and, for any $\famt\in\parball$, the map
  $\famgammam{\famt}$ is a $C^{r+1}$ diffeomorphism of $\S $ onto $\p\Omega_\famt$.  The function $\gamma$ is said to be
  a \emph{parameterization} of the family.
\end{definition}
\begin{noremark}
  We adopt the following typographical conventions for parameterizations. The symbol $\famt$ is always used to denote
  different elements of the family $\Om_\famt$.  The symbol $\apar$ denotes an arbitrary parameterization of the
  boundary of some domain $\Om$.  The symbol $s$ always denotes \emph{arc-length} parameterization of the boundary of
  some domain $\Om$.  In Section~\ref{s_proof} we introduce the \emph{Lazutkin} parameterization of the boundary of a
  domain $\Om$: it is always denoted by the symbol $x$.
\end{noremark}
\begin{definition}
  {\newcommand{\isom}{\mathcal{I}} A family $\family\Om$ is said to be \emph{isometric} (or \emph{trivial}) if there
    exists a family $\family{\isom}$ of isometries $\isom_\famt:\R^2\to\R^2$ (i.e.\ composition of a rotation and a
    translation) such that $\Omega_\famt = \isom_\famt\Omega_0$.}  A family $\family\Om$ is said to be \emph{constant}
  if $\Om_\famt = \Om_0$ for all $|\famt|\le 1$.
\end{definition}
\begin{remark}
  For a given family $\family\Om$, the parameterization $\gamma$ is, of course, not unique.  In fact, $\gamma$ and
  $\tilde\gamma$ parameterize the same family $\family\Om$ if and only if there exists a $C^1$-family of $C^{r+1}$
  circle diffeomorphisms $\apar:\parball\times\S\to\S$, such that
  $\tilde\gamma(\famt,\tilde \apar) = \gamma(\famt,\apar(\famt,\tilde \apar))$.  We call two parameterizations
  \emph{equivalent} if they correspond to the same family of domains.  Furthermore, notice that we do not consider
  families which differ by a \emph{time reparameterization} to be equivalent.
\end{remark}
We now proceed to define the main object of our work: families of isospectral domains.
\begin{definition}
  A family $\family\Om$ is said to be \emph{dynamically isospectral}\footnote{ In the literature this notion is also
    known as \emph{length-isospectrality}.} if $\LS(\Om_\famt) = \LS(\Om_{\famt'})$ for any $\famt,\famt' \in\parball$.
\end{definition}
Equipped with the above definition, we can define the \emph{dynamical spectral rigidity}
of a domain $\Om$.
\begin{definition}\label{def:DSR}
  A domain $\Om\in\cM$ is said to be \emph{dynamically spectrally rigid in $\cM$} if any dynamically isospectral family
  of domains $\family\Om$ in $\cM$ with $\Om_0 = \Om$ is an isometric family.
\end{definition}
We are going to show that if $r$ is sufficiently large, any domain that is sufficiently close to a circle and axially
symmetric is dynamically spectrally rigid in this class of domains.  We now proceed to define the class.

\begin{definition}
  We say that $\Om$ is \ $\Z_2$-\emph{symmetric} (or \emph{axially symmetric}) if there exists a reflection of the plane
  $\refl:\R^2\to\R^2$ such that $\refl\Om = \Om$.  Denote by $\cS^r$ the set of all $\Z_2$-symmetric domains in $\cD^r$.
\end{definition}

To introduce the notion of \emph{closeness to a circle}, recall that
a closed curve is a circle if and only if its curvature is constant.
\begin{definition}
  Let $\Omega\in\cD^r$ of perimeter $1$ parameterized in arc-length by $\gamma$ and let $\D_\Om$ be a disk of perimeter
  $1$ that is tangent to $\Omega$ at the point $s = 0$, parameterized in arc-length by $\gamma_{\D}$.  For $\delta > 0$,
  $\Omega$ is said to be \emph{$\delta$-close to a circle} if
  \begin{align*}
    \left\|\gamma-\gamma_\D\right\|_{C^{r+1}} \le \delta.
  \end{align*}
  A domain $\Omega\in\cD^r$ of arbitrary perimeter is said to be \emph{$\delta$-close to a circle} if its rescaling of
  perimeter $1$ is $\delta$-close to a circle.
\end{definition}
We denote by $\cD^r_\delta$ (\resp $\cS^{r}_\delta$) the set of domains in $\cD^r$ (\resp $\cS^{r}$) that are
$\delta$-close to a circle.  We are finally able to state the main result of this paper.
\begin{mthm}%
  Let $r = \rightSmoothness$; there exists $\delta > 0$ such that any domain $\Om\in\cS^r_\delta$, is dynamically
  spectrally rigid in $\cS^r_\delta$.
\end{mthm}
\begin{remark}\label{rm:spectra-equality}
  It turns out that the Laplace spectrum generically determines the length spectrum.  More precisely, assume that the
  following generic conditions are met
  \begin{enumerate}
  \item no two distinct periodic orbits have the same length;
  \item the Poincar\'e map of any periodic orbit of the associated billiard ball map (see~\eqref{def:bill-map}) is
    non-degenerate.
  \end{enumerate}
  Then we can replace the ``$\subset$'' symbol in~\eqref{eq:laplace-length} with an ``$=$'' sign (see~\cite[Chapter
  7]{PS}; indeed the same result holds in arbitrary dimension, and one can even drop the convexity assumption in the
  case of planar domains).
\end{remark}
In view of the above remark, our Main Theorem has an immediate rephrasing in terms of the spectral rigidity problem.
\begin{ocorollary}
  Let $r = \rightSmoothness$ and let $\tilde\cS^r_\delta\subset\cS^r_\delta$ be the set of domains in $\cS^r_\delta$
  that satisfy the generic assumptions listed in Remark~\ref{rm:spectra-equality}.  There exists $\delta > 0$ such that
  any $\Om\in\tilde\cS^r_\delta$ is spectrally rigid in $\tilde\cS^r_\delta$.
\end{ocorollary}
In other words, finitely smooth $\Z_2$-symmetric convex domains close to a circle are generically spectrally rigid.
\begin{remark}
  Hezari, in a recent preprint (see~\cite{He}), using the method of this paper combined with wave trace invariants of
  Guillemin--Melrose and the heat trace invariants of Zayed for the Laplacian with Robin boundary conditions, show that
  one can generalize the Dirichlet/Neumann spectral rigidity claimed in the above corollary to the case of Robin
  boundary conditions (see \cite{He} for the references).
\end{remark}
\section{A preliminary reduction}
\label{s_normalized}
It is natural to introduce a notion of normalization, which allows us to restate our result in a simpler manner; this
will be accomplished in Theorem~\ref{t_main''}.  Let $\Om\in\cS^r$; in the case that $\Om$ admits more than one axes of
symmetry, let us choose (arbitrarily) one of such axes and refer it as \emph{the} symmetry axis of $\Om$.  Since the
domain $\Om$ is convex, its symmetry axis intersects $\p\Om$ in two points. Choose (arbitrarily) one of such points: we
refer to it as \emph{the marked point of $\p\Om$}; the other point will be referred to as \emph{the auxiliary point of
  $\p\Om$}.  From now on, whenever we consider a domain $\Om$, we assume that a choice for the symmetry axis, the marked
point and the auxiliary point has been made.  Observe that once they have been chosen for $\Om_0$, then, by continuity,
they are unambiguously determined for any element of the family $\family\Om$.  Furthermore, we also assume that the
parameterization $\gamma$ of the family $\family\Om$ is such that $\gamma(\famt,0)$ is the marked point of $\Om_\famt$.

\begin{definition}
  A domain $\Omega\in\cS^r$ is said to be \emph{normalized} if the marked point of $\p\Omega$ is at the origin of $\R^2$
  and the auxiliary point lies on the positive $x$-semi-axis.  A family $\family\Om$ is said to be \emph{normalized} if
  $\Om_\famt$ is normalized for any $|\famt|\le 1$.
\end{definition}
Naturally, given a family $\family\Om$, we can always use isometries to construct an associated normalized family
$\family{\tilde\Om}$ as follows:
\begin{itemize}
\item translate each domain so that the marked point of $\p\Omega_{\tau}$ is at the origin
  of $\R^2$.
\item rotate the domain around the origin so that the auxiliary point of $\p\Omega_{\tau}$ lies on the positive horizontal semi-axis
\end{itemize}
We call $\family{\tilde\Om}$ the \emph{normalization} of the family $\family\Om$. Since $\family\Om$ is a $C^1$-family,
we gather that $\family{\tilde\Om}$ is also a $C^1$-family. Observe that, as each $\tilde\Omega_\famt$ is obtained from
$\Omega_\famt$ via an isometry, we have $\LS(\Omega_\famt) = \LS(\tilde\Omega_\famt)$. In particular, $\family\Om$ is a
dynamically isospectral family if and only if so is $\family{\tilde\Om}$.

We can now give an equivalent statement of our Main Theorem as follows:
\begin{theorem}\label{t_main''}
  Let $r = \rightSmoothness$; there exists $\delta > 0$ such that if $\family\Om$ is a normalized, dynamically
  isospectral $C^1$-family of domains in $\cS^r_\delta$, then $\family\Om$ is a constant family.
\end{theorem}

We now proceed to set up yet another equivalent statement of our Main Theorem, which will be stated as
Theorem~\ref{t_main'}.

 Given a parameterization $\gamma$ of a family $\family\Om$ in $\cS^r$, we define the
{\it infinitesimal deformation function}:
\begin{align*}
  n_\gamma(\famt,\apar) = \dotProd{\partial_\famt\gamma(\famt,\apar)}{\Norm_\gamma(\famt,\apar)},
\end{align*}
where $\dotProd\cdot\cdot$ is the usual scalar product in $\R^2$ and $\Norm_\gamma(\famt,\apar)$ is the outgoing unit
normal vector to $\p\Omega_\famt$ at the point $\gamma(\famt,\apar)$.  Observe that $n$ is continuous in $\famt$ and
$n(\famt,\cdot)\in C^r(\S,\R)$ for any $\famt\in\parball$.
By the normalization condition of $(\Om_{\famt})_{|\famt|\leq 1}$, we conclude that $n_\gamma(\famt,\cdot)$ is an even
function, \ie $n_\gamma(\famt,\apar) = n_\gamma(\famt,-\apar)$, and
 $n_\gamma(\famt,0) = 0$ for any $\famt\in\parball$.
\begin{lemma}
  Let $\family\Om$ be a family of domains in $\cD^r$ and let $\gamma(\famt,\apar)$ be a parameterization of the family. Then
  \begin{enumerate}
    \item\label{i_0is0} for any other parameterization
    $\tilde\gamma(\famt,\tilde \apar) = \gamma(\famt,\apar(\famt,\tilde \apar))$ we have
    $n_{\tilde\gamma}(\famt,\tilde\apar) = n_\gamma(\famt,\apar(\tau,\tilde\apar))$.  In particular, if for some
    $\famt\in\parball$ we have $n_\gamma(\famt,\cdot)$ is identically $0$, then $n_{\tilde\gamma}(\famt,\cdot)$ is
    identically $0$.
  \item \label{i_global} $n_\gamma(\famt,\apar) = 0$ for all $(\famt,\apar)\in\parball\times\S$ if and only if $\family\Om$ is a
    constant family.
  \end{enumerate}
\end{lemma}
\begin{proof}
  Let us fix $\famt$ and assume that $\gamma(\famt,\cdot)$ and $\tilde\gamma(\famt,\cdot)$ are two parameterization of
  $\Om_\famt$, \ie $\gamma(\famt,\apar) = \tilde\gamma(\famt,\tilde \apar(\famt,\apar))$.  Differentiation with
  respect to $\famt$ reads:
  \begin{align*}
    \p_\famt\gamma(\famt,\apar) = \p_\famt\tilde\gamma(\famt,\tilde \apar(\famt,\apar))+\p_\apar\tilde\gamma(\famt,\tilde \apar(\famt,\apar))\p_\apar\tilde \apar(\famt,\apar).
  \end{align*}
  By taking the scalar product with $N_\gamma(\famt,\apar) = N_{\tilde\gamma}(\famt,\tilde \apar(\famt,\apar))$ we
  conclude that $n_\gamma(\famt,\apar) = n_{\tilde\gamma}(\famt,\tilde \apar(\famt,\apar))$ which proves item ~\ref{i_0is0}.

  In order to prove item ~\ref{i_global}, observe that if $n_\gamma(\famt,\apar) = 0$, then the vector $\p_\famt\gamma(\famt,\apar)$ is
  necessarily a multiple of $\p_{\apar}\gamma(\famt,\apar)$; since $\gamma(\famt,\cdot)$ is assumed to be a diffeomorphism,
we  conclude that $d\gamma(\famt,\apar)$ has rank $1$
everywhere.  It follows from the Constant Rank Theorem that
the image of  $\gamma$ is a manifold of dimension $1$ which can be parameterized by $\gamma(0,\cdot)$.  This implies that
  $\Om_\famt = \Om_0$ for any $|\famt|\le 1$.
The converse holds trivially by item~\ref{i_0is0}.
\end{proof}

We can thus further restate Theorem~\ref{t_main''} (and thus our Main Theorem) as follows.
\begin{theorem}\label{t_main'}
  Let $r = \rightSmoothness$, then there exists $\delta > 0$ such that if $\family\Om$ is a normalized dynamically isospectral
  $C^1$-family of domains in $\cS^r_\delta$ then $n_\gamma = 0$ for all parameterizations $\gamma$.
\end{theorem}

\section{Billiard dynamics of \texorpdfstring{$\Z_2$}{ℤ₂}-symmetric domains}\label{s_billiard}
Let $\Om\in\cS^r$; for definiteness we fix the perimeter of its boundary to be $1$. Recall that $s$ denotes the
arc-length parameterization and that we conventionally assume that the marked point has coordinate $s = 0$; moreover,
since $\Om$ has perimeter $1$, the auxiliary point has coordinate $s = 1/2$.
We consider the billiard dynamics on $\Om$, which is described as follows: a point particle travels with constant
velocity in the interior of $\Om$; when the particle hits $\p\Om$, it bounces according to the law of optical
reflection: angle of incidence $=$ angle of reflection.  Periodic trajectories of the billiard dynamics are thus,
essentially\footnote{ A periodic trajectory and its time-reversal trace in fact the same closed geodesic; the exceptions
  are \emph{bouncing ball trajectories}, which are invariant for time-reversal and thus correspond $1$-to-$1$ to their
  closed geodesics.} in $2$-to-$1$ correspondence to closed geodesics of $\Om$.  It is customary to study the billiard
dynamics by passing to a discrete-time version of it, i.e. to a map on the canonical Poincar\'e section
$M = \p\Om\times [-1,1]$.  The first coordinate (parameterized\footnote{ We abuse notation and also denote by
  $\gamma:\S\to\R^2$ the parameterization of the single domain $\Om$.}  by $\gamma(s)$) identifies the point at which
the particle has collided with $\p\Om$ and the second coordinate $y$ equals $\cos \varphi $, where $\varphi$ is the
angle that the outgoing trajectory forms with the positively oriented tangent to $\p\Om$.
The billiard ball map $f$ on $M$ is then defined as \be \label{def:bill-map} \beal
f: \p\Om\times [-1,1]&\to \p\Om\times [-1,1]\\
(s,{y })&\mapsto (s',{y '}), \enal \ee where $s'$ is the coordinate of the point at which the trajectory emanating from
$s$ with angle $\varphi$ collides once again with $\p\Om$ and {$y '=\cos \varphi'$}, where $\varphi'$ is the angle of
incidence of the trajectory with the negatively oriented tangent to $\p\Om$ at $s'$.  The map $f$ is an exact twist
diffeomorphism which preserves the area form {$ds\wedge dy $}.
Let us denote by
\[L(s,s')=\|\gamma(s)-\gamma(s')\|\] the Euclidean distance between the two points on $\partial \Omega$ parameterized by
$s$ and $s'$.  Notice that $L$ is a generating function of the billiard ball map, \ie we have:
\begin{align*}
  \begin{cases}
    \frac{\p L}{\p s}(s,s')=-y \\
    \frac{\p L}{\p s'}(s,s')=\phantom- y'.
  \end{cases}
\end{align*}

Given $\Om_0,\Om_1\in\cS^r$, both normalized and both of length $1$, let $\gamma_0$ and $\gamma_1$ be the corresponding
arc-length parameterizations of their boundaries; let us define $\dist(\Om_0,\Om_1) = \|\gamma_0-\gamma_1\|\nc{r+1}$;
then by the above considerations we gather that for any $\delta' > 0$ there exists a $\delta$ such that if
$\dist(\Om_0,\Om_1) < \delta$, then the corresponding generating functions will also be $C^{r+1}$-close to each other on
the set $\{s\ne s'\}$; hence $\|f_{\Om_0}-f_{\Om_1}\|\nc{r} < \delta'$.

Once that we have defined the billiard map $f$, we can prove a simple but
important property of the Length Spectrum.
\begin{lemma}\label{l_spectrumNullSet}
  For any $\Om\in\cD^r$ with $r\ge 2$, $\LS({\Om})$ has zero Lebesgue measure.
\end{lemma}
\begin{proof}
  Recall that Sard's Lemma implies that the set of critical values of a real valued
  $C^r$-function defined on an $n$-dimensional manifold has zero Lebesgue
  measure provided that $r\ge n$.  For any $q$, let us define the function
  \begin{align}\label{e_lengthFunctional}
    \tilde L_q(s,y) = L(s_{0},s_{1})+L(s_1,s_2)+\cdots+L(s_{q-1},s_0),
  \end{align}
  where $(s_k,y_k) = f^k(s,y)$.  Periodic orbits of period $q$ of the billiard map correspond to critical
  points of $\tilde L_q$.  Indeed, if the $q$-tuple $(s_0,s_1,\cdots,s_{q-1})$ identifies the vertices of a
  periodic orbit, equality of the angle of reflection and the angle of incidence of the trajectory at any
  given $s_k$ implies that partial derivative of the right hand side of~\eqref{e_lengthFunctional} with
  respect to $s_k$ equals zero.  Since we can express $s_k = s_k(s,y)$, using the chain rule we conclude that
  if $(s,y)$ is a periodic point, then it is a critical point of $\tilde L_q(s,y)$; the set of lengths of such
  orbits thus corresponds to the set of critical values of $\tilde L_q$.  Since the billiard map $f$ is $C^r$
  with $r \ge 2$, we conclude that the set of lengths of periodic orbits of period $q$ has zero Lebesgue
  measure; by taking the (countable) union over $q$ we conclude that $\LS(\Om)$ has zero Lebesgue measure.
\end{proof}
\begin{iremark}
  Note that it is possible to construct (non-generic) examples of smooth domains $\Om$ whose length spectrum has
  positive Hausdorff dimension.
\end{iremark}
The above lemma will be used to impose constraints on isospectral families by means of the following immediate corollary
\begin{corollary}\label{c_isospectral}
  Let $\family\Om$ be a family of domains in $\cD^r$ and let $\Delta:\parball\to\R$ be a Darboux function\footnote{ Recall that
    a function is said to be \emph{Darboux} if it has the intermediate value property.  In particular the corollary
    applies to continuous functions. } so that $\Delta(\famt)\in\LS(\Om_\famt)$.  If $\family\Om$ is isospectral, then
  $\Delta$ is constant.
\end{corollary}
\begin{proof}
  Assume that there exists $\famt\in\parball$ such that $\Delta(\famt)\ne\Delta(0)$; since $\Delta$ is Darboux, we
  conclude that $\Delta([0,\famt])$ contains an open set.  Since the family is isospectral,
  $\Delta(\famt)\in\LS(\Om_\famt) = \LS(\Om_0)$; we conclude that $\Delta([0,\famt])\subset\LS(\Om_0)$, but this
  contradicts Lemma~\ref{l_spectrumNullSet}.
\end{proof}
In the sequel we assume $\family\Om$ to be fixed together with a parameterization $\gamma$; without risk of confusion we
thus drop all subscripts $\gamma$.  The symbol $\Om$ simply denotes an arbitrary element of the family.
Let us start with a simple case: let $\Delta_0(\famt)$ denote the perimeter of $\Om_\famt$, that is:
\begin{align*}
  \Delta_0(\famt) &= \int_0^1 \|\p_s\gamma(\famt,\apar)\|d\apar.
\end{align*}
By definition, $\Delta_0$ is continuous and $\Delta_0(\famt)\in\LS(\Om_\famt)$; we conclude by
Corollary~\ref{c_isospectral} that $\Delta_0$ is constant; hence:
\begin{align*}
  0 &= \Delta_0'(\famt) = \int_0^1\langle\p_{\famt \apar}\gamma(\famt,\apar),T(\famt,\apar)\rangle d\apar
\end{align*}
where $T(\famt,\apar) = \p_s\gamma(\famt,\apar)/\|\p_s\gamma(\famt,\apar)\|$ is the positively oriented unit tangent vector
to $\Om_\famt$ at the point $\gamma(\famt,\apar)$.  Integrating by parts we obtain:
\begin{align*}
  0 &= -\int_0^1\langle\p_\famt\gamma(\famt,\apar),\p_s T(\famt,\apar)\rangle d\apar = \\
    &= \int_0^1 \frac{n(\famt,\apar)}{\rho_{\Om_\famt}(\apar)}\frac{ds}{d\apar}d\apar
\end{align*}
where $\frac{ds}{d\apar}$ accounts for the change of variable from arc-length $s$ to $\apar$ {and, recall,
  $\rho_{\Om_\tau}$ is the radius of the curvature of $\partial\Om_\tau$}.

For any $\Om$ (parameterized by $\apar$), we define the linear functional
\begin{align}\label{e_definitionellp}
  \lfunc{\Om,0}(\nu) = \int_0^1 \frac{\nu(\apar)}{\rho_\Om(\apar)}\frac{ds}{d\apar}d\apar.
\end{align}
By our above discussion we conclude that if $\family\Om$ is isospectral, then for any $\famt\in\parball$ we have
$\lfunc{\Om_\famt,0}(n(\famt,\cdot)) = 0$.

We will now proceed to define a sequence of functionals that are related to the variation of lengths of a special class
of periodic orbits of the billiard map.  Consider a periodic orbit of period $q$ and let $p\in\Z$ denote its winding
number.\footnote{ We can define the winding number as the number of times that the associated polygon wraps around the
  boundary $\p\Om$; alternatively, by considering a lift $\hat f$ of $f$ to the universal cover $\R\times[-1,1]$, we
  have that $\hat f^q(\hat s,y) = (\hat s+p,y)$, where $p\in\Z$ defines the winding number.}  Then we define the
\emph{rotation number} of the orbits as the ratio $p/q$.  The following lemma is a simple consequence of the fact that
$\Om$ has $\Z_2$-symmetry.
\begin{lemma}\label{orbit}
  Let $\Om\in\cS^r$; for any $q\ge 2$, there exists a periodic orbit of rotation number $1/q$ passing through the marked
  point of $\p\Om$ and having maximal length among other periodic orbits passing through the marked point.  We call such
  an orbit \emph{marked symmetric maximal periodic orbit} and denote it by $S^q(\Om)$.
\end{lemma}
\begin{proof}
  Let us recall that $s$ denotes the parameterization in arc-length; we distinguish the cases of even and odd period.
  \vskip 0.1in \textbf{Case 1:} $q=2k$ is even.  We claim there exists a $q$-periodic orbit passing through the marked
  point and the auxiliary point.  Indeed, let us fix $s_0=0$ and $s_{k}=1/2$ and consider the problem of maximizing the
  function
  \begin{subequations}
    \label{e_variationalProblems}
    \begin{align}
      L_q(\svect):=\label{var1}2\sum_{i=0}^{k-1} L(s_i,s_{i+1})
    \end{align}
    where $\svect=(s_1,\cdots,s_{k-1})$ belongs to the compact set $0 = s_0\le s_1\le\cdots\le s_{k-1}\le s_k = 1/2$.
    Let $\bar\svect = (\bar{s}_1,\dots,\bar{s}_{k-1})$ be a maximum of $L_q$; observe that by the triangle inequality
    and strict convexity we have $\bar s_0 < \bar s_1 < \cdots < \bar s_{k-1} < \bar s_k$.  If we fix conventionally
    $\bar{s}_0=0$, $\bar{s}_{k}=1/2$:
    \begin{align*}
      \p_1 L(\bar{s}_i,\bar{s}_{i+1})=-\p_2L(\bar{s}_{i-1},\bar{s}_{i}),\quad i=1,\dots,k-1.
    \end{align*}

    Completing $\bar{s}_{2k-i}=-\bar{s}_i$, $i=1,\dots,k-1$, we obtain a periodic orbit of period $2k = q$, which is of
    maximal length among symmetric orbits.

\vskip 0.1in

    \textbf{Case 2:} $q=2k+1$ is odd.  We claim there exists a periodic orbit passing through the marked point and so
    that the segment $\overline{\gamma(s_k)\gamma(s_{k+1}) }$ is perpendicular to the symmetry axis.  Indeed, let us fix
    $s_0=0$ and consider the problem of maximizing the function
    \begin{align}
      \label{var2}L_q(\svect):=\sum_{i=0}^{k-1}
	2L(s_i,s_{i+1})+ L(s_k,-s_{k}),
    \end{align}
  \end{subequations}
  where $\svect=(s_1,\cdots,s_{k})$ belongs to the compact set $0 =s_0\le s_1\le\cdots\le s_{k}\le 1/2$.  Once again by
  the triangle inequality and strict convexity, the maximum is attained at a critical point
  $\bar\svect = (\bar{s}_1,\dots,\bar{s}_{k})$ so that $0 < \bar s_1 < \cdots < \bar s_{k-1} < 1/2$.  Moreover, if we
  fix conventionally $\bar{s}_0=0$, we have
  \begin{align*}
    0&=\p_1 L(\bar{s}_i,\bar{s}_{i+1})+\p_2L(\bar{s}_{i-1},\bar{s}_{i}),\quad i=1,\dots,k-1\\
    0&=\p_2 L(\bar{s}_{k-1},\bar{s}_k)+\frac{1}{2}\p_1 L(\bar{s}_k,-\bar{s}_k)-\frac{1}{2}\p_2 L(\bar{s}_k,-\bar{s}_k)\\
     &=\p_2 L(\bar{s}_{k-1},\bar{s}_k)+ \p_1 L(\bar{s}_k,-\bar{s}_k).
  \end{align*}
  Completing $\bar{s}_{2k+1-i}=-\bar{s}_i$, $i=1,\dots k-1$, we obtain a periodic orbit of period $2k+1 = q$ which is of
  maximal length amongst all symmetric orbits.
\end{proof}
Let us define $L(\famt,s,s')=\|\gamma(\famt,s)-\gamma(\famt,s')\|$.  For $q\geq 2$, let $\Lq(\famt;\svect)$ denote the
function defined in~\eqref{e_variationalProblems} for $\Om = \Omega_\famt$.  Correspondingly, let $\Delta_q(\famt)$
denote the length of the marked symmetric maximal periodic orbits of rotation number $1/q$ for the domain
$\Omega_\famt$, that is:
\begin{align*}
  \Delta_q(\famt)=\max_{\svect}\Lq(\famt;\svect).
\end{align*}
\begin{lemma}
  For any $q\geq 2$, the function $\Delta_q(\famt): \parball\to \R$ is a Lipschitz function.
\end{lemma}
\begin{proof}
 Define
  \begin{align*}
    K_q:=\max_{\famt\in \parball}\max_{\svect} \left\|{\p_\famt\Lq(\famt;\svect)}\right\|.
  \end{align*}
  Let us fix arbitrarily $\famt\in \parball$, and let $\bar{\svect}$ realize the maximum of $\Lq(\famt;\svect)$, that is
  $\Lq(\famt,\bar{\svect})=\Delta_q(\famt)$. Then $\Lq(\famt';\bar{\svect})\leq \Delta_q(\famt')$, for any
  $\famt'\in \bar{U}'$, hence
  \begin{align*}
    \Delta_q(\famt)-\Delta_q(\famt')\leq \Lq(\famt;\bar{\svect})-\Lq(\famt';\bar{\svect})\leq K_q|\famt-\famt'|
  \end{align*}
  Exchanging $\famt$ and $\famt'$ in the above inequality, we can thus conclude that:
  \begin{equation*}
    |\Delta_q(\famt)-\Delta_q(\famt')|\leq K_q|\famt-\famt'|.\qedhere
  \end{equation*}
\end{proof}
Observe that, by definition, $\Delta_q(\famt)\in\LS(\Om_\famt)$ and thus, by Corollary~\ref{c_isospectral}, if
$\family\Om$ is isospectral, then $\Delta_q$ is constant.  There is an obvious obstacle that arises when one tries to
apply the strategy that we employed above with $\Delta_0$: a priori $\Delta_q$ is not differentiable.  We therefore need
to consider a slightly more general approach.  Given a continuous function $\Phi:\parball\subset \R\to \R$, one can
define its upper (\resp lower) differential $D^+\Phi(\famt)$ (\resp $D^-\Phi(\famt)$), which is characterized as
follows: $p\in D^+\Phi(\famt)$ (\resp $p\in D^-\Phi(\famt)$) if and only if there exists a function
$G\in C^1(\parball,\R)$, such that $G'(\famt)=p$ and $G\geq \Phi$ (resp. $G\leq \Phi$) in a neighborhood of $\famt$,
with equality at $\famt$. Both $D^+\Phi(\famt)$ and $D^-\Phi(\famt)$ are convex subset of $\R$; they are both non empty
if and only if $\Phi$ is differentiable, and in this case $D^+\Phi(\famt)=D^-\Phi(\famt)=\{\Phi'(\famt)\}$.

\begin{lemma}\label{upp-dif}
  If $\bar{\svect}$ is a point realizing the maximum, i.e.:
  \begin{align*}
    \Lq(\famt;\bar\svect) = \max_{\svect} \Lq(\famt;\svect) =  \Delta_q(\famt),
  \end{align*}
  then we have $\p_{\famt} \Lq(\famt;\bar{\svect})\in D^-\Delta_q(\famt)$.
\end{lemma}
\begin{proof}
  Since $\Delta_q(\cdot)$ is the maximum of $\Lq(\cdot,\svect)$, if $\bar\svect$ realizes the maximum at $\famt$ we have
  $\Lq(\famt',\bar{\svect})\leq \Delta_q(\famt')$ for any $\famt'\in\parball$.  Noticing that $\Lq(\cdot;\bar{\svect})$
  is a $C^1$ function, by definition of lower differential, we conclude that
  $p=\p_\famt \Lq(\famt;\bar{\svect})\in D^-\Delta_q(\famt)$.
\end{proof}
\begin{iremark} Indeed, one could show that $\Delta_q(t)$ is a semi-convex function and
  \begin{align*}
    D^-\Delta_q(\famt)=\text{co}\{\p_\famt \Lq(\famt;\bar{\svect})|\Lq(\famt;\bar{\svect})=\Delta_q(\famt)\}.
  \end{align*}
\end{iremark}

Now let $\Om\in\cS^r$ parameterized by $\apar$ and assume we fixed $S^q(\Om) = (\xi_q^k,\varphi_q^k)_{k=0}^{q-1}$ a
maximal marked symmetric periodic orbit of rotation number $1/q$; then we define the functional $\ell_{\Om,q}$ as
follows: for any continuous function $\nu:\S\to\R$ we let
\begin{align}\label{e_definitionEllq}
  \ell_{\Om,q}(\nu) : = \sum_{k=0}^{q-1}\nu(\apar_q^k)\sin\varphi_q^k.
\end{align}
\begin{iremark}
  These functionals can, of course, be defined for any periodic orbit
  (rather than only for marked symmetric maximal orbits).  Since we will not use
  non-symmetric orbits for the proof of our Main Theorem, we find simpler to use the above definition.
\end{iremark}
\begin{proposition}\label{deform}
  Let $\family\Om$ be an isospectral family, then for any $\famt\in\parball$,
  $q \ge 2$ and having fixed arbitrarily $\bar S^q_\famt$ a maximal marked
  symmetric periodic orbit for $\Omega_{\famt}$, we have
  $\ell_{\Om_\famt,q}(n(\famt,\cdot)) = 0$.
\end{proposition}
\begin{proof}
  Let us fix $\famt\in\parball$ arbitrarily.  To ease our notation let us write $S^q = S^q_\famt$.  By assumption we
  have that the point $\bar\svect$ that corresponds to $\bar S^q$ is a maximum, \ie
  $\Delta_q(\famt)=\max_{\svect}\Lq(\famt;\svect)=L^q(\famt,\bar\svect)$.  Then by Lemma~\ref{upp-dif}:
  \begin{align*}
    \p_\famt \Lq(\famt;\bar{\svect})\in D^-\Delta_q(\famt)=\{0\}.
  \end{align*}
  In particular for $\bar S^q = (\bar s_q^k,\bar \varphi_q^k)_{k=0}^{q-1}$, since
  $  \Lq(\famt;\bar{\svect})=\sum_{k=0}^{q-1}L(\famt,\bar s_q^k,\bar s_q^{k+1})$
  and observing that
\begin{align*}
  {\p_\famt}L(\famt,s,s')
  &= {\p_\famt}\|\gamma(\famt,s)-\gamma(\famt,s')\|\\
  &=\frac{\gamma(\famt,s)-\gamma(\famt,s')}{\|\gamma(\famt,s)-\gamma(\famt,s')\|} \cdot%
    [\p_\famt\gamma(\famt,s)-\p_\famt\gamma(\famt,s')]
  \end{align*}
  we get
  \begin{align*}
    0 & = {\p_\famt} \Lq(\famt;\bar{\svect}) %
        = {\p_\famt}\sum_{k=0}^{q-1}L(\famt,\bar s_q^k,\bar s_q^{k+1})\\
      & = \sum_{k = 0}^{q-1}
        \left[\frac{\gamma(\famt,\bar
        s_q^k)-\gamma(\famt,\bar s_q^{k-1})}{\|\gamma(\famt,\bar s_q^k)-\gamma(\famt,\bar s_q^{k-1})\|}-\frac{\gamma(\famt,\bar s_q^{k+1})-\gamma(\famt,\bar s_q^k)}{\|\gamma(\famt,\bar s_q^{k+1})-\gamma(\famt,\bar s_q^k)\|}
        \right]\cdot\p_\famt\gamma(\famt,\bar s_q^k)    \\
      & = \sum_{k=0}^{q-1} 2\sin\bar\varphi_{q}^k\Norm(\famt,\bar s_q^k)\cdot\p_\famt\gamma(\famt,\bar s_q^k)\\
      & = 2 \sum_{k=0}^{q-1} n(\famt,\bar s_q^k)\sin\bar \varphi_q^k,
  \end{align*}
  which concludes the proof.
\end{proof}
\begin{remark}
  Let $S=\{(\famt,{\p_\famt}\Lq(\famt;\svect))| \p_{\svect} \Lq(\famt,\svect)=0\}, U=\parball$, then $S\subset T^*U$ is a
  Lagrangian submanifold.  Let $\pi: T^*U\to U$ denote the natural projection, then $\Lq(t;\cdot)$ is a Morse function
  (critical points are non-degenerate) if and only if $\famt$ is a regular value of the map $\pi|_{S}$. The set $U_1$ of
  such values is an open set, and by Sard's theorem, it has full measure. Furthermore, the set $U_0\subset U_1$ of
  $\famt$ such that $\Lq(\famt;\cdot)$ is an excellent Morse function (Morse function whose critical points have
  pairwise distinct critical values) is also an open subset of full measure. For any $\famt_0\in X_0$, the critical
  points of $\Lq(\famt;\cdot)$ depends smoothly on $\famt$ within a sufficiently small neighborhood of $\famt_0$. Hence
  we have
  \begin{align*}
    \Delta'_q(\famt)=2\ell_{S^q}(n(\famt,\cdot)),\quad \famt\in U_0
  \end{align*}
  for an arbitrary (\ie not necessarily isospectral) $C^2$ deformation $\gamma$.
\end{remark}

We now define conventionally the additional functional $\lfunc{\Om,1}(\nu)$ as the evaluation of the function $\nu$ at the
marked point $s = 0$, that is we simply let
\begin{align*}
  \lfunc{\Om,1}(\nu) = \nu(0).
\end{align*}
Observe that if $\family\Om$ is a normalized family, then the marked point is fixed at the origin and, therefore,
$\lfunc{\Om_\famt,1}(n(\famt,\cdot)) = 0$.  We summarize our findings in the following statement.
\begin{corollary}\label{c_deform}
  Let $\family\Om$ be a normalized isospectral family, then for any $q\ge 0$ we have
  $\lfunc{\Om_\tau,q}(n_\gamma(\famt,\cdot)) = 0$ for any $\famt\in\parball$.
\end{corollary}

Let us now define the space of $C^r$-smooth even functions
\begin{align*}
  \defspace = \{\nu\in C^{r}(\S)\st \nu(\apar) = \nu(-\apar)\}.
\end{align*}
We then define the \emph{linearized isospectral operator} $\lop:\defspace\to \R^\N$:
\begin{align}\label{eq:LOO-def}
  \lop_\Om\nu = \left(\lfunc{\Om,0}(\nu),\lfunc{\Om,1}(\nu),\cdots,\lfunc{\Om,q}(\nu),\cdots\right).
\end{align}
In fact, $\lop$ has range in $\ell^\infty$, by definition of the functionals $\lfunc{\Om,q}$, since by~\cite[Lemma
8]{ADK}, there exists some $C > 0$ so that for any $q \ge 2$ we have $\sin\varphi_q^k\le C/q$.

We now prove that our Main Theorem is implied by the following statement
\begin{theorem} \label{t_main}%
  Let $r = \rightSmoothness$; there exists $\delta > 0$ so that the operator
  $\lop_\Om:{C^r_{\text{sym}}}\to \ell^\infty$ is injective for any $\Om\in\cS^r_\delta$.
\end{theorem}
\begin{proof}[Proof of Theorem~\ref{t_main'}]
  Assume that $\delta$ is sufficiently small so that Theorem~\ref{t_main} holds. Suppose by contradiction that for some
  $\famt\in\parball$, we have $n_\gamma(\famt,\cdot)$ is not identically zero; hence, by Theorem~\ref{t_main} we
  conclude that there exists $q$ so that $\lfunc{\Om,q}(n_\gamma(\famt,\cdot))\ne0$; this contradicts
  Corollary~\ref{c_deform}.
\end{proof}
The rest of this paper is devoted to the proof of Theorem~\ref{t_main}.

\section{Proof of Theorem~\ref{t_main}}\label{s_proof}
Let us introduce some useful notations.
\subsection{Lazutkin coordinates} We first define a convenient parameterization of $\Om$, which is known as the \emph{Lazutkin parameterization}
(see~\cite{Lazutkin}).  Recall that the symbol $s$ denotes parameterization by arc-length; then we define the Lazutkin
parameterization, which will always be denoted by the symbol $x$, as follows:
\begin{align}
  \label{e_lazutkin}%
  x(s) &= C_\El\,\int_0^s
         \,\rho(s')^{-2/3}\ ds', &\text{where }\ C_\El
  &=\left[\int_{\partial\Om}\rho(s')^{-2/3}ds'\right]^{-1}.
\end{align}
We also introduce \emph{Lazutkin weight} as the positive function:
\begin{align} \label{correction-function}%
  \mu(x)=\frac{1}{2C_\El\rho(x)^{1/3}}.
\end{align}
The main advantage of this parametrization is that dynamical quantities related to marked symmetric (maximal) orbits
have a particularly simple form with respect to the variable $x$.
\begin{lemma}\label{l_Lazutkin}
  Assume $r \ge 8$; for any $\eps > 0$ sufficiently small, there exists $\delta > 0$ so that for any
  $\Om\in\cS^r_\delta$ there exist $C^{r-4}$ real-valued functions $\alpha(x)$ and $\beta(x)$ so that $\alpha$ is an odd
  function, $\beta$ is even, $\|\alpha\|\nc{r-4},\|\beta\|\nc{r-4} < \eps$ and, for any marked symmetric (maximal)
  $q$-periodic orbit $(x_q^0,\cdots,x_q^{q-1})$:
  \begin{subequations}\label{e_periodicLazutkinDynamics}
    \begin{align}\label{e_periodicLazutkinDynamics-x}
      x_q^k &= k/q + \frac{\alpha(k/q)}{q^2}+\eps O( q^{-4}).
    \end{align}
    Moreover, if $\varphi_q^k$ denotes the angle of reflection of the trajectory at the $k$-th collision, we have:
  \begin{align} \label{phi_k}
    \varphi_q^k &=\frac{\mu(x_q^k)}{q} \left (1+ \frac{\beta(k/q)}{q^2} + \varepsilon O(q^{-4}) \right )
  \end{align}
\end{subequations}
\end{lemma}

The proof of the above Lemma is given in Appendix~\ref{s_lazutkin}; it suggests that the Lazutkin parameterization is particularly well suited to study the functionals $\lfunc q$'s.

\subsection{The linearized map modified by the Lazutkin weight} It is more natural to define the auxiliary sequence
  of functionals
  \begin{align*}
    \tlfunc q(u) = \lfunc{q}(\mu\inv u)
  \end{align*}
and correspondingly define
\begin{align}\label{eq:tLOO-def}
  \tilde\lop u = \left(\tlfunc 0(u),\tlfunc{1}(u),\cdots,\tlfunc{q}(u),\cdots\right).
\end{align}
Observe since $\mu$ does not vanish, the injectivity of $\tilde\lop$ is equivalent to the injectivity of $\lop$.
However, the operator $\tilde\lop$ turns out to be more convenient to study.
It is, in fact, immediate to check (using the explicit formula~\eqref{e_lazutkin} and~\eqref{e_definitionellp}) that:
\begin{align*}
  \tlfunc0(u) &= 2\int_0^1 u(x)dx,
\end{align*}
\ie, $\tlfunc0$ is proportional to the averaging functional with respect to Lebesgue measure.  On the other
hand, $\tlfunc1(u) = \mu\inv(0)u(0)$ is the evaluation of $\mu\inv u$ at the marked point.  In
Appendix~\ref{Appendix-with-Hezari} (joint with H. Hezari) we study the properties of the functionals
$\tlfunc{k}$ which will be used in the rest of this section.

\subsection{Mapping structure of the linearized map $\tilde\lop$}  Recall that $C^r_\symm$ denotes the space of even $C^r$-functions of
$\T^1$; define the projector $P_*:C^r_\symm\to C^r_{*,\symm}$, where $C^r_{*,\symm}$ is the space of even, zero average
$C^r$-functions of $\T^1$,i.e.
\begin{align*}
  P_*\nl = \nl-\int_0^1 \nl(x)dx,
\end{align*}
where $dx$ is the Lebesgue measure with respect to the Lazutkin parameter $x$. Let $\Ls$ be the space of even, zero average, $L^1$ functions of $\T^1$, i.e.
\begin{align*}
  \Ls &= \left\{\nl\in L^1(\T^1)\st \nl(x) = \nl(-x), \int \nl(x)dx = 0\right\} \\
      &= \left\{\nl\in L^1(\T^1)\st \nl(x) = \sum_{j \ge 1}\nlf_j e_j \right\},
\end{align*}
where $\nlf_j$ denote its Fourier coefficients in the basis $B$.  We now proceed to define a space of \emph{admissible
  functions}: for $3 < \expo < 4$, define the subspace
\begin{align*}
  \admX&=\{\nl\in \Ls\st \lim_{j\to\infty}j^\gamma|\nlf_j| = 0\};
\end{align*}
equipped with the norm:
\begin{align*}
  \|\nl\|_\expo =\max_{j\ge1} j^\gamma |\nlf_j|.
\end{align*}
The space $(\admX\,,\|\cdot\|_\expo)$ is a (separable) Banach space.
\begin{remark}
  Because of our constraints on $\gamma$, we conclude that $C^3_{*,\symm}\subset \admX \subset C_{*,\symm}^2$,
  whence the functionals
  \begin{align*}
    \tlfunc q(n)=\tlfunc q\left(\sum_{j=1}^{\infty}n_je_j\right)=
    \sum_{j=1}^{\infty}n_j\tlfunc q(e_j).
  \end{align*}
  are well-defined on $\admX$, since the Fourier series converges uniformly.
\end{remark}
Notice that with our choice for the parameter $r$ we have $r > \gamma$.

Let $\ell^\infty_* = \{\binf = (a_i)_{i\ge 0}\in\ell^\infty \st a_0 = 0\}$ and let us introduce the subspace
\begin{align*}
  h_{*,\expo} = \{\binf = (a_i)_{i\ge 0}\in\ell^\infty_*\st \lim_{j\to\infty}j^\gamma a_j = 0\}
\end{align*}
equipped with the norm $|\binf|\nsob = \max_{j \ge0} j^\gamma|a_j|$.

We now state the main technical result of this section; then we show how Theorem~\ref{t_main} follows from this result
and finally provide its proof.
\begin{lemma}\label{l_decomposition}
  There exist linearly independent vectors $\binf_l,\binf_\bullet\in\ell^\infty\setminus h_{*,\expo}$ so that, for any
  $\Om\in\cS^r_\delta$ with $r = \rightSmoothness$, the operator $\tilde\lop:\defspace\to\ell^\infty$ can be
  decomposed as follows:
  \begin{align*}
    \tilde\lop = \binf_l \tlfunc0 + \left[\binf_\bullet \tlfunc\bullet + \tlopr\right] P_*,
  \end{align*}
  where $\tlopr:\admX\to h_{*,\expo}$ is an invertible operator provided that $\delta$ is sufficiently small.
\end{lemma}
\begin{proof}[{\bf Proof of Theorem~\ref{t_main}}]
  By assumptions, the vectors $\binf_l$, $\binf_\bullet$ and $\tlopr P_*\nl$ are linearly independent for any
  $\nl\in\defspace$.  Hence, if $\nl\in\ker\tlop$, then necessarily $\tlfunc0(\nl) = \tlfunc\bullet(\nl) = 0$ and
  $\tlopr P_*(\nl) = 0$.  Now, by definition, if $\tlfunc 0(\nl) = 0$, then $\nl = P_*\nl \in C^r_{*,\symm}$.  Since
  $\tlopr$ is injective and $\tlopr\nl = 0$ we thus conclude that $\nl = 0$.
\end{proof}
\begin{proof}[{\bf Proof of Lemma~\ref{l_decomposition}}]
  Let us first decompose
  \begin{align*}
    \tlop = \tlop((1-P_*)+P_*) = \tlop(1-P_*)+ \tlop_*P_*
  \end{align*}
  where $\tlop_*$ is the restriction of $\tlop$ on $C^r_{*,\symm}$.  Observe that, by definition $(1-P_*)\nl$ is the
  constant function equal to $\tlfunc0(\nl)$; we can thus set
  \begin{align*}
    \binf_l := \tlop(1) = (1,1,\cdots,1,\cdots) + \eps(0,0,O(1),\cdots, O(q^{-2}),\cdots)\in\ell^\infty.
  \end{align*}
  We thus conclude that $\tlop = \binf_l\tlfunc0+\tlop_*P_*$.  Let us now define
  \begin{align*}
    \binf_\bullet = (0,0,1/4,\cdots, 1/q^2,\cdots);
  \end{align*}
  and let $\tlopr = \tlop_* - \binf_\bullet\tlfunc\bullet$, where $\tlfunc\bullet$ is defined
  in~\eqref{e_definitionEllBullet} so that
  \begin{align*}
    \tlfunc\bullet(u)=\tlfunc\bullet\left(\sum_{j\geq 1}\hat{u}_je_j\right)=
    \sum_{j\geq 1}\hat{u}_j(\tilde\sigma_j+\beta_j-2\pi j\alpha_j),
  \end{align*}
  where $\alpha_j$ and $\beta_j$ are the Fourier coefficients of $\alpha$ and $\beta$, $\tilde\sigma_j$ is defined
  in~\eqref{e_definitionEllBullet} and $|\alpha_j|,|\beta_j|, |\tilde\sigma_j| = \eps O(j^{-r'})$, where $r' = r-4$ is
  the smoothness of $\alpha$ and $\beta$, provided $\delta$ is sufficiently small.  Clearly, $\binf_\bullet$ and
  $\binf_l$ are linearly independent and neither of them belongs to $h_{*,\gamma}$ since $\gamma > 3$.  We now consider
  the operator $\tlopr$; let $(\tilde T_{qj})_{q,j} = (\tlfunc q(e_j))_{q,j}$ denote the matrix representation of
  $\tlopr$ in the canonical basis.  We will in fact show that
  \begin{align}\label{e_invertibility}
    \|\tlopr-\id\|_\expo &< 1 &\text{if $\eps$ is sufficiently small}
  \end{align}
  where $\|\cdot\|_\expo$ is the operator norm from $(\admX,\|\cdot\|_\expo)$ to
  $(h_{*,\expo},|\cdot|_\expo)$. For any linear operator $\mathcal L:\admX\to h_{*,\expo}$
  identified by the matrix $(L_{qj})_{q,j}$, we have
  \begin{align*}
    \|\mathcal{L}\|_\expo = \sup_q\sum_{j > 0}q^\expo j^{-\expo}|L_{qj}|.
  \end{align*}

  For $q = 1$, we have by definition $\tilde T_{1j} = \tlfunc1(e_j) = 1$; on the other hand, for $q\ge2$,
  Lemma~\ref{lem:l(e)-expansion} yields the expression:
  \begin{align}
        (\tilde{\mathcal T}_{*, R})_{qj} = \left (1+\Sfo{0}{q}+ \frac{\beta_0}{q^2} \right ) \dec qj + \mathcal R_{qj},
        \label{remainder-inv}
  \end{align}
  where $\mathcal R_{qj}$ is the (matrix representation of the) remainder term.  First, we claim that the operator
  $\Delta:\admX\to h_{*,\expo}$, identified by the matrix $\dec qj$, satisfies the following bound:
  \begin{align}\label{e_firstInvertibility}
    \|\Delta- \id\|_\expo < \zeta(3)-1<0.21,
  \end{align}
  where $\zeta$ is the Riemann zeta function.  In particular,\footnote{ The value $\zeta(3)$ is also known as the
    \emph{Ap\'ery's constant}.} $\|\Delta\|_\expo<\zeta(3) < 1.21$ and has bounded inverse.  In fact, by definition, the
  norm $\|\Delta-\id\|_\expo$ is given\footnote{ Here $\delta_{qj}$ is the usual Kronecker delta notation.}  by:
  \begin{align*}
    \|\Delta-\id\|_\expo &= \sup_{q} q^\expo\sum_{j > 0}j^{-\expo}(\dec qj-\delta_{qj})
    = \sup_{q} \left[q^{\expo}\sum_{j > 0}j^{-\expo}\dec qj\right] -1\\
                         &= \sup_{q} q^{\expo}\sum_{s > 0}(sq)^{-\expo}-1 \le \sum_{s > 0}s^{-3} -1= \zeta(3)-1
  \end{align*}
  since $\expo > 2$.  In particular $\|\Delta\| < \zeta(3)$. This shows that if $\Delta'$ is the operator defined by
  \begin{align*}
    (\Delta')_{1j} &=0\\
    q \geq 2: \quad (\Delta')_{qj} &=\left (\Sfo{0}{q}+ \frac{\beta_0}{q^2} \right ) \Delta_{qj},
  \end{align*}
  then by~\eqref{e_SfoEstimates} we conclude
  \begin{align*}
    \| \Delta'\|_\expo \leq \left ( \frac{(\pi +\eps)^2}{24} + \frac{\eps}{4} \right ) \zeta(3).
  \end{align*}
  By choosing $\eps>0$ small enough we can make sure that
  \begin{align*}
    \| \Delta' \|_\expo \leq 0.51.
  \end{align*}
  We thus conclude that
  \begin{align*}
    \|\Delta+\Delta'-\id\|_\expo < \|\Delta+\Delta'-\id\|_\expo < 0.8.
  \end{align*}
  Now
  using the above expression, we prove~\eqref{e_invertibility} by showing that if $\eps$ is sufficiently small,
  $\|\mathcal R\|_\expo < C\eps$.  Recall that
  \begin{align*}
    \mathcal R_{qj} = \frac1{q^2}\sum_{\substack{s\in\Z\setminus\{0\}\\ sq\ne j}}({q^2\Sfo{sq-j}{q}+\beta_{sq-j}+2\pi ij\alpha_{sq-j}})+\eps O(j^2q^{-4}).
  \end{align*}
  Let us first check the contribution to the norm of the term $\eps O(j^2q^{-4})$:
  \begin{align*}
    \eps\sup_q \sum_{j > 0}\cO(j^{2-\expo}q^{\expo-4}).
  \end{align*}
  Since $\expo > 3$, the sum on $j$ converges, and since $\expo < 4$, the sequence to converges to $0$ as $q\to\infty$.
  We conclude that this term can be made as small as needed by choosing $\eps$ sufficiently small.
  Next, we deal with the sum: since $\alpha,\ \beta$ are $C^{r'}$-smooth functions where $r' = r-4 = 4$ and
  by~\eqref{e_SfoEstimates} we gather:
  \begin{align*}
    q^{\expo}j^{-\expo}\frac1{q^2}\left| \sum_{\substack{s\in\Z\setminus\{0\}\\sq\ne j}} ({q^2\Sfo{sq-j}{q}+\beta_{sq-j}+2\pi ij\alpha_{sq-j}})
    \right| %
    \\ < \eps O(q^{\expo-2})\sum_{\substack{s\in\Z\setminus\{0\}\\ sq\ne j}}\frac 1{|sq-j|^{r'}j^{\expo-1}}.
  \end{align*}
  Let us now estimate the sum over $j$ of the above expression.
  \begin{align*}
    q^{\expo-2}\sum_s \sum_{j > 0} &\frac 1{|sq-j|^{r'}j^{\expo-1}} \\
                                   &= q^{\expo-2}\sum_s\left[ \sum_{0 < j < [sq]^+} + \sum_{j > [sq]^+}\right] \frac 1{|sq-j|^{r'}j^{\expo-1}} = \I + \II,
  \end{align*}
  where $[sq]^+ = \max\{0,sq\}$.  Let us first consider the term $\II$; we have:
  \begin{align*}
    \II \le q^{\expo-2} \sum_{s}\frac1{|sq|^{\expo-1}}\sum_{j > sq}\frac1{|sq-j|^{r'}} =
    \frac Cq\sum_s |s|^{1-\expo} < \frac Cq.
  \end{align*}
  Then we consider term $\I$; let us write:
  \begin{align*}
    \I = q^{\expo-2}\sum_s\left[\sum_{0 < j < sq/2}+\sum_{sq/2\le j < sq}\right] \frac 1{|sq-j|^{r'}j^{\expo-1}}= \I'+\I''.
  \end{align*}
  Then
  \begin{align*}
    \I' &\le q^{\expo-2}\sum_{s > 0}\sum_{0 < j < sq/2}\frac{2^{r'}}{|sq|^{r'}}\frac1{j^{\expo-1}} < \frac C{q^{r'-\expo}}\sum_{s >
          0}\sum_{j > 0}\frac1{s^{r'}j^{1+\expo}} < \frac C{q^{r'-\expo}},\\
    \I''&\le q^{\expo-2}\sum_{s > 0}\sum_{sq/2 < j < sq}\frac1{|sq-j|^{r'}}\frac{2^{\expo-1}}{(sq)^{\expo-1}} < \frac Cq\sum_{s >
          0}\sum_{p > 0}\frac1{p^{r'}s^{\expo-1}} < \frac Cq.
  \end{align*}
  which allows to conclude since $\expo <  r' = 4$.
\end{proof}
\section{A finite dimensional reduction for arbitrary symmetric domains}\label{s_finiteDim}

In this section, we outline the proof of a more general result, which holds for $\Z_2$-symmetric domains that
are not necessarily close to a circle.  For any $\nu\in C^{r}_\symm$, let
 \begin{align*}
   \nu(x)&=\sum_{k\ge 0}\hat{\nu}_k \cos (2\pi kx)
 \end{align*}
 be its Fourier expansion.  Define the finite co-dimension space
 \begin{align*}
   C^{r}_{q_0,\symm} = \{\nu\in C^{r}_\symm  \st
   \hat{\nu}_0=\dots=\hat{\nu}_{q_0-1}=0\}.
 \end{align*}
 Moreover, define the space of admissible function
 \begin{align*}
   L_{q_0,\text{sym}}^1=\{\nu\in L^1(\T^1),\,\nu(x)=\sum_{j\geq q_0}\hat{\nu}_j \cos 2\pi jx\}
 \end{align*}
 and
 \begin{align*}
 X_{q_0,\gamma}=\{\nu\in L_{q_0,\text{sym}}^1, \text{s.t.} \lim_{j\to\infty}j^{\gamma}|\hat{\nu}_j|=0\}
 \end{align*}
 equipped with the norm $\|\nu\|_{q_0,\gamma}=\max_{j\geq q_0}j^{\gamma}|\hat{\nu}_j|$.

Define the operator $\tilde{\lop}_{q_0}:C^{r}_{q_0,\symm}\to\ell^\infty$ as
 \begin{align*}
   \tilde{\lop}_{q_0}(u) = (\tilde{\ell}_{q_0}(u),\tilde{\ell}_{q_0+1}(u),\cdots,\tilde{\ell}_{q}(u),\cdots).
 \end{align*}
 Then we can show the following
 \begin{theorem} \label{t_main_q}%
   Let $r = \rightSmoothness$, then for any domain $\Om \in S^{r}$ there exists $q_0=q_0(\Om)$ such that the operator
   $\tilde{\lop}_{q_0}:C^{r}_{\text{sym}}(q_0)\to \ell^\infty$ is injective.
 \end{theorem}
 \begin{proof}
   The proof follows the one of Theorem~\ref{t_main}: first we need to establish the existence of good Lazutkin
   coordinates; in the main case, since $\delta$ was sufficiently small, we could find Lazutkin coordinates of order $5$
   on the whole phase space; in our present context this is not guaranteed, and one can only find them in a neighborhood
   of $y = 0$, that is, for sufficiently large $q$.  Then, similar to Lemma~\ref{l_decomposition}, we have a
   decomposition for the operator $\tilde{T}_{q_0}$ as follows:
   \begin{align*}
     \tilde{T}_{q_0}=b_{q_0}\tilde{\ell}_*+\tilde{T}_{q_0,R}
   \end{align*}
   with $b_{q_0}=(1/q^2)_{q\geq q_0}$, and $\tilde{T}_{q_0,R}:X_{q_0,\gamma}\to h_{\gamma}$ is an invertible operator,
   where $h_{\gamma}=\{a=(a_i)_{i\geq 1}\in \ell^{\infty},\, \lim_{j\to\infty} j^{\gamma}|a_j|=0\}$ equipped with the
   norm $\|a\|_{\gamma}=\max_{j\geq 1}j^{\gamma}|a_j|$. The proof of the invertibility of $\tilde{T}_{q_0,R}$ follows
   that of $\tilde{T}_{*,R}$ where, without the condition of being close to a circle, $O(\eps)$ is replaced by $O(1)$
   everywhere, as defined in~\eqref{remainder-inv}, we have instead
   \begin{align*}
     \|(R_{q,j})_{q,j\geq q_0}\|_{\gamma}<C/q_0
   \end{align*}
   which hence can be made arbitrarily small when $q_0$ is large enough.
 \end{proof}

 In Definition~\ref{def:DSR} we define dynamically spectrally rigid domains. In our setting the space $\cM=\cS^r$
 consists of $\Z_2$-symmetric domains.

\begin{corollary}
  For $r\geq 8$ and any non-dynamically spectrally rigid domain $\Om\in \cS^r$ there is a linear subspace
  $\mathcal N(\Om) \ni \Om $ of dimension at most $q_0(\Om)$ such that any isospectral family
  $(\Om_\tau)_{|\tau|\le 1},\ \Om_0=\Om$ is tangent to $\mathcal N(\Om)$ for $\tau=0$.
\end{corollary}
\begin{proof}
Consider a family of isospectral deformations $(\Om_\tau)_{|\tau|\le 1}$. Let
$\nu \in C^r_{\symm}$ be the associated function at $\tau=0$. By Corollary \ref{c_deform}
we have that $\tilde{\lop}_{q_0}(\nu)=0$. Decompose $\nu$ into
$\nu=\nu_{q_0}+\nu_{q_0}^\perp$, where $\nu_{q_0}^\perp \in C^r_{q_0, \symm}$
is the natural projection of $\nu$ in $C^r_{\symm}$ onto this subspace
and  $\nu_{q_0}$ is the complement given by a trigonometric polynomial
of degree $<q_0$. Since $\tilde{\lop}_{q_0}(\nu)=0$, we have
\begin{align*}
\tilde{\lop}_{q_0}(\nu_{q_0}^\perp)=-\tilde{\lop}_{q_0}(\nu_{q_0}).
\end{align*}
Therefore, by Theorem \ref{t_main_q} for each $\nu_{q_0}$ there is at most one $\nu_{q_0}^\perp$ solving this equation.
If the linear spaces $\tilde{\lop}_{q_0}(C^r_{q_0, \symm})$ and the image of its orthogonal complement under
$\tilde{\lop}_{q_0}$ intersect, then any isospectral family $(\Om_\tau)_{|\tau|\le 1}$ is tangent to it.
\end{proof}


\section{Concluding remarks}\label{s_conclusions}
In this paper we proved dynamical spectral rigidity of convex domains which are $\Z_2$-symmetric and close to a circle;
it is indeed natural to ask if the same result holds if one drops some of our assumptions

\textsc{$\Z_2$-symmetry:} the main challenge in removing the symmetry assumption with our strategy is that one would
need to find another sequence of periodic orbit which generate linear functionals that are linearly independent of the
ones corresponding to maximal periodic orbits.  This appears to be a non-trivial task, since, as $q$ increases (by the
results of Appendix~\ref{s_Lazutkin}), the dynamics is closer and closer (to any order) to the dynamics of a billiard in
a disk.  On the other hand, due to the symmetries of the disk, all linear functionals corresponding to orbits of the
same rotation number would be linearly dependent.

\textsc{Closeness to a circle:} in Section~\ref{s_finiteDim} we showed that to prove our result for domains that are not
  necessarily close to a circle it would suffice to find a finite number of linearly independent functionals for small
  periods.  However, we do not have a priori any control on such orbits, and the general strategy is unclear.

  \textsc{Convexity:} as we mentioned in the introduction, our result is an analog of Guillemin-Kazhdan (see
  ~\cite{GKaz}) for $\Z_2$-symmetric domains close to the circle.  However, from the dynamical point of view,
  more natural analogs of geodesic flows on surfaces of negative curvature are dispersing billiards:
  ``\textit{Are dispersing billiards spectrally rigid?}''.  One possible approach to prove this statement
  would be to introduce and study the linearized isospectral operator analogous to~\eqref{eq:LOO-def}.


\appendix

\section{Lazutkin coordinates}\label{s_Lazutkin}
\newcommand{\smo}{s}
\newcommand{\sml}{s}
\subsection{An abstract setting}%
Let $\A = \S\times I$, where $I\subset\R$ is a compact interval; let us assume without loss of generality that
$I = [0,1]$.  We denote by $(x,y)$ the natural coordinates in $\S\times I$.  Let $F\in C^\smo(\A,\A)$ be a monotone
orientation preserving twist diffeomorphism which leaves invariant both boundary components of $\A$.  We further assume
that the circle $\{y = 0\}$ is the union of \emph{fixed} points, \ie
\begin{align*}
  F(x,0) &= (x,0)&\text{for all } x&\in\T.
\end{align*}
To avoid complications in the exposition let us assume that the other boundary
component $\{y = 1\}$ is also fixed by $F$.  Furthermore, denote with $\hat F$
a lift of $F$ to $\hat\A = \R\times I$ so that
\begin{align*}
  \hat F(X,0) &= (X,0)&\text{for all } X&\in\R.
\end{align*}
To fix ideas, we further assume that for any $X\in\R$, we have $\hat F(X,1) = (X+1,1)$.

\begin{definition}
  Let $N\in\N$; a function $\Lazf\in C^\sml(\hat\A,\R)$ is called a \emph{$C^{\sml}$-Lazutkin function of order $N$} if
  \begin{enumerate}
  \item for any $y\in I$ the map $\Lazf(\cdot,y):\R\to\R$ is a lift of an orientation preserving diffeomorphism of
    $\S$ so that $\Lazf(0,y) = 0$.
  \item there exists $\rema\in C^\sml(\hat\A,\R)$ with $\rema(x,y) = \rema_*(x,y)y^{N+1}$ and $\rema_*\in C^\sml(\hat\A,\R)$ so
    that:
    \begin{align}\label{e_lazutkinFunction}
      \Lazf\circ \hat F - 2 \Lazf + \Lazf\circ \hat F\inv = \rema.
    \end{align}
  \end{enumerate}
\end{definition}
Observe that, by definition $\rema(x,0) = 0$ and our assumptions on $\hat F$ and $\Lazf$ imply that
$\rema(x,1) = 0$ for any $x\in\T^1$.

Given a Lazutkin function $\Lazf$, define the real function $y_\Lazf:\hat A\to\R$:
\begin{align}\label{e_definitionYll}
  \yll = \frac12\left[\Lazf\circ \hat F - \Lazf\circ \hat F\inv\right].
\end{align}
Since $F$ is a twist map, $\yll(x,\cdot)$ is strictly increasing for any fixed $x$.  Moreover, by our assumptions on $F$
we gather that $\yll(x,0) = 0$ and $\yll(x,1) = 1$.  We conclude that, for any $x\in\R$, the function $y_\Lazf(x,\cdot)$
is a diffeomorphism of $I$ onto itself.  Hence, $(X_\Lazf,Y_\Lazf)$ (where we set $X_\Lazf = \Lazf(x,y)$) are good
coordinates on $\hat\A = \R\times I$, which factor to $\A$ as $(\xll,\yll)$.  Let us denote with $\Coc_{\Lazf}$ the
change of variables $\Coc_\Lazf:(x,y)\mapsto (\xll,\yll)$; notice that by design, $\Coc_\Lazf$ leaves invariant the
boundary components of $\A$.  A simple computation shows that $\Coc_\Lazf$ conjugates the map $F$ to
$\tilde F:(\xi,\eta)\mapsto(\xi^+,\eta^+)$ where:
\begin{align*}
    &\begin{cases}
      \xi^+ &= \xi+\eta+ \rema_\Lazf(\xi,\eta) \text{ mod }1\\
      \eta^+ &= \eta+ \rema_\Lazf(\xi,\eta)+ \rema_\Lazf(\xi^+,\eta^+),
  \end{cases}
            &\text{where }\rema_\Lazf &= \frac{1}{2}\rema\circ\Coc_\Lazf\inv.
\end{align*}
In particular, we can write $ \rema_{\Lazf}(\xi,\eta) = \rema_{\Lazf,*}(\xi,\eta)\eta^{N+1}$ and
$\rema_{\Lazf}(\xi,0) = \rema_{\Lazf}(\xi,1) = 0$) for any $\xi\in\T$.


\begin{lemma}[Properties of the Normal Form]\label{l_normalFormProperties}
  Assume $s\ge2$ and let $\rema\in C^\smo(\A,\R)$
be so that $\rema(\xi,0) = \rema(\xi,1) = 0$ with $\|\rema\|\nc\smo$
  sufficiently small.  Then there is a unique map $F_\rema\in C^\smo(\A,\A)$ that fixes $\{\eta = 0\}$ and so that
  $F_\rema(\xi,\eta) = (\xi^+,\eta^+)$ where:
  \begin{subequations}
    \begin{align}\label{e_LazutkinNormalForm-x}
      \begin{cases}
        \xi^+(\xi,\eta) = \xi+\eta+\rema(\xi,\eta) \pmod1\\
        \eta^+(\xi,\eta) = \eta+\rema(\xi,\eta)+\rema(\xi^+,\eta^+).
      \end{cases}
    \end{align}
    Moreover, $F_\rema(\xi,1) = (\xi,1)$ and $F_\rema\inv(\xi,\eta) = (\xi^-,\eta^-)$, where
    \begin{align}\label{e_LazutkinNormalForm-inv}
      \begin{cases}
        \xi^-(\xi,\eta)= \xi-\eta+\rema(\xi,\eta) \pmod1\\
        \eta^-(\xi,\eta)= \eta-\rema(\xi,\eta)-\rema(\xi^-,\eta^-).
      \end{cases}
    \end{align}
  \end{subequations}
\end{lemma}
\begin{proof}
  Since $\rema$ is fixed, observe that $\xi^+$ is an explicit well-defined function; we thus only need to show that
  there exists a unique $\eta^+(\xi,\eta)$ satisfying the required relation with initial condition  $\eta^+(\xi,\eta = 0) = 0$. It follows from the Implicit Function Theorem applying to the relation
  \[\eta + \rema(\xi,\eta)+\rema (\xi^+(\xi,\eta),\eta^+)-\eta^+=0\]
  since $\p_{2}\rema-1$ is not zero.
Then observe that $\eta^+(\xi,1) = 1$ satisfies the functional equation,
and thus, by uniqueness, we conclude that $F_\rema$ fixes
the boundaries $\{\eta=0\}$ and $\{\eta=1\}$.

  The expression~\eqref{e_LazutkinNormalForm-inv} follows from simple algebraic manipulations
  of~\eqref{e_LazutkinNormalForm-x} and it is left to the reader.
\end{proof}
We call~\eqref{e_LazutkinNormalForm-x} the \emph{Lazutkin normal form with remainder $\rema$ of order $N$ and class
  $C^\smo$} if $\rema(x,y) = \rema_*(x,y)y^{N+1}$ and $\rema_*\in C^s(\xA,\R)$.  Coordinates $(\xi,\eta)$ are said to be
\emph{Lazutkin coordinates of order $N$ for $F$} if they conjugate $F$ to a Lazutkin normal form of order $N$.

The next lemma constitutes the main result of this section: it gives sufficient conditions to find Lazutkin coordinates
of any order.
\begin{lemma}\label{l_lazutkinFunctions}
  Let $\smo\ge 3$ and assume that for some $N \ge 1$ the dynamics $F$ is described in the coordinates $(x,y)$ by the
  Lazutkin Normal Form~\eqref{e_LazutkinNormalForm-x} with remainder $\rema$ of order $N$ and class $C^s$, \ie
  $\rema(x,y) = \rema_*(x,y)y^{N+1}$ with $\rema_*\in C^s(\hat \A,\R)$.  Then, if $\|\rema_*\|\nc\smo$ is sufficiently
  small, there exists a $C^\smo$ change of variables $\Psi:(x,y)\mapsto(\bar x,\bar y)$ so that $(\bar x,\bar y)$ are
  Lazutkin coordinates of order $N+1$ with remainder $\bar \rema(\bar x,\bar y) = \bar\rema_*(x,y)y^{N+2}$ with
  $\bar\rema_*\in C^{\smo-1}(\hat\A,\R)$ and, moreover, for some universal $C_*$ we have
  \begin{enumerate}
  \item $\Psi(x,y) = (x + \Psi_0(x,y)y^{N-1},y + \Psi_1(x,y)y^{N}),$ where $\|\Psi_i\|\nc\smo < C_*\|\rema_*\|_{C^\smo}$
    for $i = 0,1$;
  \item $\|\bar \rema_*\|\nc{\smo-1}\le C_* \|\rema_*\|\nc\smo$
  \end{enumerate}
\end{lemma}
\begin{proof}
  The key to the proof is to find suitable Lazutkin functions; to simplify the exposition it is convenient to treat
  separately the case $N = 1$ and $N > 1$.  First, let us set some convenient notation: let
  \begin{align*}
    \rema(x,y) =  \rema_*(x,y)y^{N+1} &=  \rema_0(x)y^{N+1}+\tilde \rema(x,y)y^{N+2}.
  \end{align*}
  where $ \rema_0(x) = \rema_*(x,0)$ is $C^\smo$ and, by definition, $\tilde \rema$ is a $C^{\smo-1}$ function so that
  $\|\tilde \rema\|\nc{\smo-1}\le\|\rema_*\|\nc\smo$.

  \textbf{Case $N = 1$:} we proceed to construct a Lazutkin function $\Lazf$ of order $2$: let us make the \emph{ansatz}
  $\Lazf(x,y) = \lazf(x)$, where $\lazf$ solves the differential equation
  \begin{align}\label{e_ansatz1}
    2\lazf'(x) \rema_0(x)+\lazf''(x) = 0
  \end{align}
  with boundary conditions $\lazf(0) = 0$ and $\lazf(1) = 1$.  Let us prove that $\Lazf$ is a Lazutkin function: first
  of all, elementary ODE considerations imply that $\lazf$ is $C^{s+2}$ and that $\lazf' > 0$, which, together with the
  boundary conditions, implies that $x\mapsto\lazf(x)$ is the lift of a circle diffeomorphism.  Moreover, by
  construction we have that $\|\lazf- \id\|\nc\smo\le\|\rema_*\|\nc\smo$.  We thus need to show that $\Lazf$
  satisfies~\eqref{e_lazutkinFunction} with a remainder of order $2$.  Fix $x\in\T$; then for any $x'\in\T$ we write
  \begin{align*}
    \lazf(x') = \lazf(x)+\lazf'(x)(x'-x)+\lazf''(x)(x'-x)^2 + \tilde\lazf(x,x')(x'-x)^3,
  \end{align*}
  where $\tilde\lazf$ is a $C^{\smo-1}$ function so that $\|\tilde\lazf\|_{C^{\smo-1}}\le\|\lazf'''\|_{C^{\smo-1}}$
  and by definition of $\lazf$, we have
$\|\lazf'''\|_{C^{\smo-1}} < C_*\|\rema_*\|_{C^\smo}$.
First, we apply the above
  expansion to $x' = x^+$; notice that:
  \begin{align*}
    (x^+-x)\phantom{^1} &= y+ \rema_0(x)y^2+\tilde  \rema_1(x,y)y^3\\
    (x^+-x)^2 &= y^2 + \tilde  \rema_2(x,y)y^3\\
    (x^+-x)^3 &= \tilde  \rema_3(x,y)y^3,
  \end{align*}
  where $\tilde \rema_1,\tilde \rema_2,\tilde \rema_3$ are $C^{s-1}$ functions so that
  $\|\tilde \rema_i\|_{C^{s-1}}\le C_*\|\rema_*\|_{C^s}$, provided that $C_*$ is sufficiently large.  Then we conclude
  that
  \begin{align*}
    \lazf(x^+) &= \lazf(x)+\lazf'(x)[y+ \rema_0(x)y^2]+\frac{\lazf''(x)}2y^2+\bar\lazf^+(x,y)y^ 3\\
               &= \lazf(x) + \lazf'(x) y + \left(\lazf'(x) \rema_0(x)+\frac{\lazf''(x)}2\right)y^2+\bar\lazf^+(x,y)y^3
                 \intertext{and by~\eqref{e_ansatz1} we conclude}
               &= \lazf(x) + \lazf'(x) y + \bar\lazf^+(x,y)y^3,
  \end{align*}
  where $\bar\lazf^+$ is $C^{\smo-1}$ and $\|\bar\lazf^+\|_{C^{s-1}}\le\|\rema_*\|_{C^{s}}$.
  Applying the same argument to $x' = x^-$ we obtain similarly
  \begin{align*}
    \lazf(x^-) &= \lazf(x) -\lazf'(x) y + \bar\lazf^-(x,y)y^3,
  \end{align*}
  where $\bar\lazf^-$ has the same properties as $\bar\lazf^+$.  We thus found that $\Lazf(x,y) = \lazf(x)$ satisfies~\eqref{e_lazutkinFunction} with
the $C^{\smo-1}$ remainder
 $$
\bar \rema(x,y) = (\bar\lazf^+(x,y)+\bar\lazf^-(x,y))y^{3}
$$
of order $2$, which concludes the proof of the case $N = 1$.

  \textbf{Case $N > 1$:} this case is similar to the previous one, but simpler.  In this case we make the ansatz
  $\Lazf(x,y) = x+\lazf(x)y^{N-1}$, where we assume that
  \begin{align}\label{e_anzatzK}
    \lazf''(x) = -2 \rema_0(x)
  \end{align}
  with boundary conditions $\lazf(0) = \lazf(1) = 0$.  Then, if $\|\rema_*\|_{C^0}$ is sufficiently small
  $\p_x\Lazf(x,y) = 1+\lazf'(x)y^{N-1} > 0$, thus $\Lazf(\cdot,y)$ is the lift of a circle diffeomorphism.  Moreover, it
  is immediate to observe that $\|\Lazf-\id\|_{C^s} = \|\lazf\|_{C^s}y^{N -1}\le\|\rema_*\|\nc\smo y^{N-1}$.  Moreover,
  we can write:
  \begin{align*}
    \Lazf(x^+,y^+) &= x+y+ \rema_0(x)y^{N+1}+\lazf(x)y^{N-1}+\lazf'(x)y^N\\
                   &\phantom = +\frac{\lazf''(x)}2y^{N+1}+\bar\lazf^+(x,y)y^{N+2} \\
                   &=  x+y+\lazf(x)y^{N-1}+ \lazf'(x)y^{N}\\
                   &\phantom = +\left( \rema_0(x)+\frac{\lazf''(x)}2\right)y^{N+1}+\bar\lazf^+(x,y)y^{N+2},\\
    \intertext{and again using our ansatz}
                   &=  x+y+\lazf(x)y^{N-1}+ \lazf'(x)y^{N} +\bar\lazf^+(x,y)y^{N+2}
  \end{align*}
  Correspondingly:
  \begin{align}
    \Lazf(x^-,y^-)  &=  x-y+\lazf(x)y^{N-1}- \lazf'(x)y^{N} +\bar\lazf^-(x,y)y^{N+2}.
  \end{align}
  We thus conclude that $\Lazf(x,y) = x+\lazf(x)y^{N-1}$ is a Lazutkin function of order $N+1$ with $C^{\smo-1}$
  remainder which we call $\bar  \rema(x,y)$.  The estimates on the norms follow
  from arguments that are similar to the case $N = 1$ and are left to the reader.
\end{proof}

 Let us define the \emph{symmetrized annulus}
$\xA = \T\times[-1,1]$ together with its universal cover $\xhA = \R\times [-1,1]$ and define the idempotent map
$\invo:(x,y)\mapsto(x,-y)$.

We now proceed to state a refinement of the above lemma, which holds under some additional assumptions.  With a little
abuse of terminology, let us say that a function $h\in C^\smo([0,1],\R)$ is even (\resp odd) if it has even (\resp odd)
extension $\eex h$ to $[-1,1]$ is of class $C^\smo$.  We say that the remainder $\rema$ is \emph{even} if for any fixed
$x\in\T$ the function $\rema(x,\cdot):I\to\R$ is even.  As we will see in the following section, this assumption will be
satisfied in our setting.    The following lemma is an immediate consequence of the definitions and of
Lemma~\ref{l_normalFormProperties}.
\begin{lemma}\label{l_involution}
  Assume $s\ge 2$ and let $\rema\in C^s(\A,\R)$ be an even function satisfying
  the hypotheses of
  Lemma~\ref{l_normalFormProperties}; then there exists a diffeomorphism
  $\xF_\rema:\xA\to\xA$ so that
  $\left.\xF_\rema\right|_\A = F_\rema$ and $\xF\circ\invo = \invo\circ(\xF)\inv$.\qed
\end{lemma}
In other terms, the map $\xF$ admits an involution, which is given by $\invo$.  Observe moreover that if $\rema$ is an
even function and $\rema(x,y) = \rema_*(x,y)y^{N+1}$, then we can always assume that $N$ is odd.  This leads to the
following version of Lemma~\ref{l_lazutkinFunctions}.
\begin{lemma}\label{l_lazutkinEvenFunctions}
  Let $\smo\ge 4$ and assume that for some $N =2K+1$ with $K\ge 0$, the dynamics $F$ is described in the coordinates
  $(x,y)$ by the Lazutkin Normal Form~\eqref{e_LazutkinNormalForm-x} with \emph{even} remainder $\rema$ of order $N$ and
  class $C^s$.  Then, if $\|\rema_*\|\nc\smo$ is sufficiently small, there exists a $C^\smo$ change of variables
  $\Psi:(x,y)\mapsto(\bar x,\bar y)$  so that $(\bar x,\bar y)$ are Lazutkin coordinates of order $N+2$ with \emph{even}
  remainder $\bar \rema$ of class $C^{\smo-2}$ so that, for some universal
  $C_*$:
  \begin{enumerate}
  \item $\Psi(x,y) = (x + \Psi_0(x,y)y^{N-1},y + \Psi_1(x,y)y^{N})$ where $\|\Psi_i\|\nc\smo < C_*\|\rema_*\|_{C^\smo}$
    for $i = 0,1$;
  \item $\|\bar \rema_*\|\nc{\smo-2}\le C_* \|\rema_*\|\nc\smo$
  \end{enumerate}
\end{lemma}
\begin{proof}
  The proof is analogous to the one given for Lemma~\ref{l_lazutkinFunctions}; in fact following the same steps we
  obtain the needed result, except the fact that we only know that the new remainder $\bar\rema$ is an even function
  only when expressed in the old coordinates $(x,y)$.  Hence, the only thing that we need to show is that $\bar\rema$ is
  an even function also when expressed in the new coordinates $(\bar x,\bar y) = \Coc(x,y)$.  This however, follows by
  Lemma~\ref{l_involution} and the definition~\eqref{e_definitionYll}, which in turn imply that the coordinate change
  $\Coc$ commutes with the involution $\invo$, \ie $\Coc\circ\invo = \invo\circ\Coc$.  This concludes the proof.
\end{proof}
Let us summarize in words the results of this section: for any $N > 1$, if we can find Lazutkin coordinates of order $1$
with remainder that is both sufficiently smooth and sufficiently small, then we can find Lazutkin coordinates of order
$N$ and we have good control on the change of variables.
The following lemma gives sufficient conditions for existence of Lazutkin coordinates of order $1$ with even remainder.
\begin{lemma}\label{l_lemmaX}
  Assume there exists $\xF:\xA\to\xA$ of class $C^r$ so that $\xF\circ\invo = \invo\circ(\xF)\inv$ and
  $F = \left.\xF\right|_\A$.  Then there exist Lazutkin coordinates of order $1$ with even remainder of class $C^{r-2}$.
\end{lemma}
\begin{proof}
  Let us write $\xF(x,y) = (\eex x,\eex y)$ so that $x^{+}(x,y) = \eex x(x,y)$ and $x^-(x,y) = \eex x(x,-y)$.  Let us
  then choose $\Lazf(x,y) = x$; hence:
  \begin{align*}
    \Lazf\circ F(x,y)&-2\Lazf(x,y)+\Lazf\circ F\inv(x,y)\\
                     &= \eex x(x,y)-2x+\eex x(x,-y) = \rema(x,y)
  \end{align*}
  which is by construction an even function with $\rema(x,0) = 0$.  We conclude that $\rema$ is an even remainder of
  order $1$, which allows to construct Lazutkin coordinates of order $1$ with even remainder.
\end{proof}
\subsection{Application to billiard dynamics}\label{s_lazutkin}
In this section we apply the results of the previous section to the billiard map $f$ corresponding to some domain
$\Om\in\cD^r$ and prove Lemma~\ref{l_Lazutkin}

 Let us assume $\Om$ to be of perimeter $1$.  As we
mentioned in Section~\ref{s_billiard}, if $s\in\T$ denotes the arc-length parameterization of the boundary and
$y = \cos\varphi\in[-1,1]$, where $\varphi\in[0,\pi]$ is the angle of the outgoing trajectory with the positively
oriented tangent vector, the map $f_\Om = f:\T\times[-1,1]\to\T\times[-1,1]\to$ is a monotone twist map of class $C^r$.
Moreover, it is clear by the definition that $f$ fixes the boundary components $y = -1$ and $y = 1$ and that
$f(s,[-1,1])$ twists only once around the annulus.  In summary, we can apply to the map $f$ the results described in an
abstract setting in the previous section; the first step is to find Lazutkin coordinates of order $1$.

Let $s$ identify a point on the boundary of $\Om$, and let $\Line_\varphi$ be the oriented line passing through the
point $s$ with angle $\varphi\in[-\pi,\pi]$ measured counterclockwise from the positively oriented tangent to $\Om$.
Since $\Om$ is strictly convex, for any $\varphi\in(-\pi,\pi)\setminus\{0\}$ there exists a unique other point of
intersection of $\Line_\varphi$ with $\Om$; let us denote this point by $\eex s(s,\varphi)$ (and extend by continuity
$\eex s(s,0) = \eex s(s,\pi) = \eex s(s,-\pi) = s)$).  Let us moreover denote with $\eex\varphi(s,\varphi)$ the angle
between $\Line_\varphi$ and the positively oriented tangent vector to $\Om$ at $s'$, also measured counterclockwise.
Notice that if $\varphi\in[0,\pi]$, by our construction, we have $\eex s(s,\varphi) = s^+(s,\varphi)$ and
$\eex\varphi(s,\varphi) = \varphi^+(s,\varphi)$.  Observe moreover that, since $\Line_\varphi = \Line_{\pi+\varphi}$
(with reversed orientation) we also have $s^-(s,\varphi) = \eex s(s,-\varphi)$ and
$\varphi^-(s,\varphi) = -\eex\varphi(s,-\varphi)$.  In other words, if we define
$\xF(s,\varphi) = (\eex s,\eex \varphi)$, this map satisfies the hypotheses of Lemma~\ref{l_lemmaX}; hence we conclude
that there exist Lazutkin coordinates of order $1$ for the billiard with even remainder $\rema$ of class $C^{r-2}$.

As noticed earlier, if two domains $\Om$ and $\Om'$ of length $1$ are $C^{r+1}$-close, then the corresponding billiard
maps $f_\Om$ and $f_{\Om'}$ are $C^r$-close; since $\rema(s,\varphi) = \eex s(s,\varphi)-2s+\eex s(s,-\varphi)$, we
gather that the corresponding remainders will be $C^r$-close.  Since for the circle we have that $\rema = 0$, we
conclude that $\rema$ can be made arbitrarily small in the $C^r$ topology provided that $\Om$ is sufficiently
$C^{r+1}$-close to a circle.

We are now in the position to give the main result of this appendix.

\begin{proof}[\textbf{Proof of Lemma \ref{l_Lazutkin}}]

  An explicit computation (but see also~\cite[(1.1--2)]{Lazutkin}) allows to write
  \begin{align*}
    \eex s (s,\varphi) = s+ 2\rho(s)\varphi + a(s,\varphi)\varphi^2
  \end{align*}
  where $a(s,\varphi)$ is a $C^{r-2}$-smooth function.  As pointed out in Lemma \ref{l_lemmaX}, we can always find Lazutkin coordinates of
  order $1$ with even remainder;
 moreover, for any $\eps > 0$ we can choose $\delta > 0$ so that if
  $\Om\in\cD^r_\delta$, then $\|\rema(s,\varphi)\|\nc{r-2} < \eps$.

  We can then apply Lemma~\ref{l_lazutkinEvenFunctions} and obtain Lazutkin coordinates of order $3$, which we denote
  with $(x,y)$.  It is easy to check that $x$ is (unsurprisingly) given by the Lazutkin parameterization defined
  in~\eqref{e_lazutkin}.

   Using the explicit formula~\eqref{e_definitionYll} we conclude that:\footnote{ Notice that
    although the $x$ coordinate is given by the standard Lazutkin parameterization, the $y$ variable is not the
    usual Lazutkin coordinate $y$ defined in~\cite[(1.3)]{Lazutkin}}
  \begin{align*}
    y(s,\varphi) &= 2C_\Om\rho^{1/3}(s)\varphi\left[ 1 + \beta_0(s)\varphi^2+c\right],
  \end{align*}
  where $\beta_0\in C^{r-2}$ with $\|\beta_0\|\nc{r-2} < \eps$ and $O\nc{r-4}(\varphi^k)$ denotes a $C^{r-4}$ function
  of $(s,\varphi)$ whose $C^{r-4}$ norm is bounded by $C \varphi^k$ for some $C > 0$.  Inverting the above expression we
  obtain
  \begin{align}\label{e_varphiFunctionOfY}
    \varphi(x,y) = \mu(x)y\left[1+\beta_1(x) y^2+\eps O\nc{r-4}(y^4)\right]
  \end{align}

  where recall $\mu$ was defined in~\eqref{correction-function} and $\beta_1$ satisfies the same estimates as $\beta_0$.
  Finally, observe that the remainder in the coordinates $(x,y)$ is $O\nc{r-4}(\eps)$.
  By applying once again Lemma~\ref{l_lazutkinEvenFunctions} we can now obtain Lazutkin coordinates of order $5$, which
  we denote with $(\bar x,\bar y)$.  By construction such coordinates conjugate the dynamics to
  \begin{align}\label{e_dynamicsO6}
    \begin{cases}
      \bar x^+ = \bar x +\bar y + \eps\cO\nc{r-6}(y^6) \\
      \bar y^+ = \bar y + \eps\cO\nc{r-6}( y^6).
    \end{cases}
  \end{align}
  and moreover there exist $C^{r-4}$ functions $\alpha(\bar x)$ and $\beta_2(\bar x)$ so that
  \begin{align}\label{e_alphaBeta}
    x &= \bar x+\alpha(\bar x)\bar y^2
    & y &= \bar y\left[1+\beta_2(\bar x)\bar y^2 + \eps O\nc{r-6}(\bar y^4)\right].
  \end{align}
  Moreover, since the remainder in the coordinates $(x,y)$ is $O\nc{r-4}(\eps)$, we conclude that $\alpha$ and $\beta_2$
  are $O\nc{r-4}(\eps)$.   Let us now consider a marked symmetric periodic
  orbit of rotation number $1/q$ which we denote with $((\bar x_q^k,\bar y_q^k))_{k \in\{0,\cdots,q-1\}}$. Recall that
  by convention $\bar x_q^0 = 0$.  Then~\eqref{e_dynamicsO6} yields:
  \begin{align*}
    \bar x_q^k &=  \frac kq+\eps\cO(q^{-4})
    & \bar y_q^k &=  \frac 1q+\eps\cO(q^{-5}).
  \end{align*}
  Combining the above with~\eqref{e_alphaBeta} and~\eqref{e_dynamicsO6} yields
  \begin{align*}
    x_q^k = \frac kq + \frac{\alpha(k/q)}{q^2}+\eps O(q^{-4})
  \end{align*}
  which is~\eqref{e_periodicLazutkinDynamics-x}.
  We also see that by~\eqref{e_periodicLazutkinDynamics-x} and
  estimates~\eqref{e_varphiFunctionOfY} and~\eqref{e_alphaBeta} we gather that for any $\eps >0$ there exists
  $\delta >0$ so that for any $\Omega \in \cS_\delta^r$ there exists $\alpha \in C^{r-4}$ odd and $\beta_3 \in C^{r-4}$
  even such that
  \begin{align*}
    x_q^k &=\frac{k}{q}+ \frac{\alpha(k/q)}{q^2} + \varepsilon O(q^{-4}),\\
    \varphi_q^k &=\frac{\mu(x_q^k)}{q} \left (1+ \frac{\beta_3(k/q)}{q^2} + \varepsilon O(q^{-4}) \right ),
  \end{align*}
  where $\alpha=O_{r-4}(\eps)$ and $\beta_3 =O_{r-4}(\eps)$.

\end{proof}

\newpage

\section{Representation of the linearized problem in the Fourier basis}\label{Appendix-with-Hezari}
\begin{center} \small By Jacopo De Simoi, Hamid Hezari, Vadim Kaloshin, and Qiaoling Wei \end{center}
\vspace{0.1in}

The original version of this paper (arXiv:1606.00230v1) contained an error which was found and corrected by
H. Hezari (see also~\cite{He}).  This resulted in modifying the form of the operator $\tilde {\mathcal T}$,
statement and proof of Lemma \ref{lem:l(e)-expansion}, estimates in Section \ref{s_proof} in the proof of
Theorem \ref{t_main}.  These modifications naturally introduce a new term $S_q(x)$ \eqref{e_S-function} that
one has to keep track of in the proof of the main theorem.  In this joint appendix, the corrections are
combined with the original argument to produce a complete proof of Lemma \ref{lem:l(e)-expansion}.  This lemma
is used in the proof of Theorem \ref{t_main}.
\\

To study the linear functionals $\tilde{\ell_q}$ we first need to understand the expressions
$\frac{\sin\varphi_q^k}{\mu(x_q^k)}$. We recall that $\mu(x)=\frac{1}{2C_\El\rho(x)^{1/3}}$.  One can easily
check that if $\partial\Om$ is a circle, then $\mu(x)$ is a constant equal to $\pi$.  Observe moreover that
for any $\eps >0$, there exists $\delta >0$ so that for any $\Omega \in \cS_\delta^r$ we have
$\|\mu(x) -\pi \|_{C^{r-1}} < \eps$. First, using Lemma \ref{l_Lazutkin}, we will show that
\begin{align}\label{e_periodicLazutkinDynamics-phi}
    \frac{\sin\varphi_q^k}{\mu(x_q^k)}
    &=\frac1q\left[1+\frac{\beta(k/q)}{q^2} +S_q\left(\frac kq \right) + \eps O( q^{-4})\right],
  \end{align}
where
\begin{align}\label{e_S-function}
S_q(x)=
\frac{\sin \left( \mu(x)/q\right)}{\mu(x)/q}-1.
    \end{align}

To do this, we first simplify $\mu(x^k_q)$ using the asymptotic of
  $x^k_q$ provided in \eqref{e_periodicLazutkinDynamics-x}, the mean value theorem, an the fact that $\mu(x)$ is uniformly bounded from below, to obtain
\begin{equation} \label{mu}
\mu(x_q^k) = \mu \left (\frac{k}{q}+ \frac{\alpha(k/q)}{q^2} \right  ) \left ( 1 +   \varepsilon O(q^{-4})\right ).
\end{equation}
Plugging the above expression into the equation of $\varphi^k_q$  provided in \eqref{phi_k}, we get
\begin{align*}
  \varphi_q^k &= \frac{1}{q} \mu \left (\frac{k}{q}+ \frac{\alpha(k/q)}{q^2} \right  )\left (1+ \frac{\beta(k/q)}{q^2} + \varepsilon O(q^{-4}) \right ).
\end{align*}
Next, we take $\sin$ of both sides and use the mean value theorem again and also the lower bound $\sin( \mu(x) /q) \geq C/q$ to get
\begin{align*}
  \sin \varphi_q^k &=  \sin \left ( \frac{1}{q} \mu \left (\frac{k}{q}+ \frac{\alpha(k/q)}{q^2} \right  )\left (1+ \frac{\beta(k/q)}{q^2} \right ) \right ) \left(1+ \varepsilon O(q^{-4}) \right).
\end{align*}
 Dividing by $\mu(x^k_q)$, using
\eqref{mu}, and also multiplying and dividing by $1+\frac{\beta(k/q)}{q^2}$, we obtain
\begin{align*}
  \frac{ \sin \varphi_q^k}{\mu(x^k_q)} =  \frac{1}{q} \frac{\sin \left ( \frac{1}{q} \mu \left (\frac{k}{q}+ \frac{\alpha(k/q)}{q^2} \right  )\left (1+ \frac{\beta(k/q)}{q^2} \right ) \right )}{\frac{1}{q}\mu \left (\frac{k}{q}+ \frac{\alpha(k/q)}{q^2} \right  ) \left (1+\frac{\beta(k/q)}{q^2} \right ) } \left (1 + \frac{\beta(k/q)}{q^2}+  \varepsilon O(q^{-4}) \right ),
\end{align*}
On the other hand by the mean value theorem and the properties $\| \alpha \|_{C^0}, \| \beta \|_{C^0}, \| \mu' \|_{C^0} \leq C \eps$, we have
$$ \frac{1}{q} \mu \left (\frac{k}{q}+ \frac{\alpha(k/q)}{q^2} \right  )\left (1+ \frac{\beta(k/q)}{q^2} \right ) = \frac{1}{q} \mu \left (\frac{k}{q} \right) + \eps O(q^{-3})$$

Using this and the mean value theorem once more and also using $\frac{d}{dy} \left ( \frac{\sin y}{y} \right )=O(y)$,
we finally get
\begin{align*}
  \frac{ \sin \varphi_q^k}{\mu(x^k_q)} &=  \frac{1}{q} \left (\frac{ \sin \left ( \frac{\mu \left ({k}/{q} \right  )}{q}  \right )}{ \frac{\mu \left ({k}/{q} \right  )}{q} }+ \frac{\beta(k/q)}{q^2} + \varepsilon O(q^{-4}) \right )
\end{align*}
We now conclude by writing this expression in the form
\begin{equation}\label{Sq}
  \frac{ \sin \varphi_q^k}{\mu(x^k_q)} =  \frac{1}{q} \left (1+ \frac{\beta(k/q)}{q^2} + \varepsilon O(q^{-4}) \right ) + \frac{1}{q} S_q(\ifrac{k}{q}),
\end{equation}
where $S_q(x)$ is defined in~\eqref{e_S-function}; this implies~\eqref{e_periodicLazutkinDynamics-phi} as claimed.

Next, in order to obtain a useful
expression for $\tlfunc q$ for $q\ge2$, recall the definition
\begin{align*}
  \lfunc{q}(\nu) &= \frac1q\sum_{k = 0}^{q-1}\nu(x_q^k)\sin(\varphi_q^k),
\end{align*}
where $(x_q^k)_{k\in\{0,\cdots,q-1\}}$ is a marked symmetric maximal periodic orbit and
$(\varphi_q^k)_{k\in\{0,\cdots,q-1\}}$ are the corresponding reflection angles.
Using~\eqref{e_periodicLazutkinDynamics} we conclude:
\begin{align}\notag
  \tlfunc{q}(\nl) &= \lfunc{q}\left(\mu\inv\nl\right) = \frac1q\sum_{k = 0}^{q-1}u(x_q^k)\frac{\sin(\varphi_q^k)}{\mu(x_q^k)}, \\
                  &= \frac1q\sum_{k = 0}^{q-1}\nl\left(k/q +
                    \frac{\alpha(k/q)}{q^2}\right)\left(1+\frac{\beta(k/q)}{q^2}\right)+\Sop_q(\nl)+\eps O( q^{-4})\label{e_tlfuncExpression},
\end{align}
where $\Sop_q$ is the following functional
\begin{align*}
  \Sop _q (u) = \frac{1}{q} \sum_{k=0}^{q-1} S_q\left (\ifrac{k}{q} \right ) u \left (\ifrac{k}{q} + \frac{\alpha(\ifrac{k}{q})}{q^2} \right ).
\end{align*}
that is obtained by substituting in the expression for $\tlfunc{q}(\nl)$ the
expression~\eqref{e_periodicLazutkinDynamics-phi} and summing over $k$.

Let us now introduce the Fourier basis $B = (e_j)_{j\ge 0}$ of even real functions of the circle in the Lazutkin
parameterization $(e_j=\cos 2\pi jx)_{j\ge0}$ together with the convenient notation
\begin{align*}
  \dt_{q|j}=\begin{cases} 1& \textup{if $q\,|\,j$}\\ 0 &\text{otherwise.}
  \end{cases}
\end{align*}
Finally, let us write $\alpha(x) = \sum_{k\in\Z}\alpha_k\exp(ikx)$
(and similarly for $\beta$); by the parity properties of  $\alpha$ and $\beta$ we conclude:
\begin{align*}
  \alpha(x)&=\sum_{k\geq 1} 2i\alpha_k\sin 2\pi kx,
& \beta(x) &=\sum_{k\geq 0} 2\beta_k\cos 2\pi kx
\end{align*}
where $\alpha_k{ = -\alpha_{-k}}$ is purely imaginary and $\beta_k{ = \beta_{-k}}$ is real.

We now analyze the functional $\Sop_q(e_j)$; let us first record the following properties of the function $S_q$.  For
any $x$ we have
\begin{subequations}
  \begin{align}
    | S_q(x) | &\le \frac{\mu^2(x)}{6q^2}   \leq \frac{(\pi +\eps)^2}{6q^2},\label{e_firstPropertyS}\\
    | S^{(r)}_q(x) | &= \eps O( q^{-2} ), \quad r \geq 1.\label{e_secondPropertyS}
  \end{align}
\end{subequations}
  Next, by definition:
    \begin{align*}
      \Sop_q (e_j) &= \frac{1}{q} \sum_{k=0}^{q-1} S_q \left ( \ifrac{k}{q} \right)\cos \left ( 2 \pi j \left ( \ifrac{k}{q} + \frac{\alpha(k/q)}{q^2} \right )\right ) .
    \end{align*}
By the mean value theorem and~\eqref{e_firstPropertyS}, we can write:
  \begin{align*}
    \Sop_q (e_j) = \frac{1}{q} \sum_{k=0}^{q-1} \cos \left ( \frac{2 \pi j k}{q}\right ) S_q \left ( \frac{k}{q} \right) + \eps O\left(\frac{j}{q^4} \right).
  \end{align*}
  We then plug in the Fourier series of $S_q(x)$, given by
  \begin{align*}
    S_q(x) = \sum_{p \in \mathbb Z} \Sfo{p}{q} e^{2 \pi i p x},
  \end{align*}
 and obtain the expression
 \begin{align}\label{e_definitionSigmaFunctional}
   \Sop_q (e_j) &= \sum_{s \in \Z} \Sfo{sq-j}{q} + \eps O\left(\frac{j^2}{q^4}\right).
 \end{align}
 Notice that, since $S_q$ is even, we have $\Sfo pq = \Sfo{-p}q$; moreover, by the properties of $\Sop_q(x)$ we have
 \begin{subequations}\label{e_SfoEstimates}
   \begin{align}
     |\Sfo{0}{q}| & \le  \frac{(\pi +\eps)^2}{6q^2},\\
     | \sigma_p(q)  | &= \eps O\left( \frac{1}{p^{r}q^2} \right)\quad \text {for $p \neq 0$,}
   \end{align}
 \end{subequations}
where the second equation follows from integration by parts.

We can now prove the following convenient expansion
\begin{lemma}\label{lem:l(e)-expansion}
For all $q \geq 2$ and $j \geq 1$, one has
\begin{align*}
  \tilde {\ell}_q(e_j) = \left(1+ \Sfo{0}{q} + \frac{\beta_0}{q^2}\right) \delta_{q|j} + \frac{\tlfunc\bullet(e_j)}{q^2} +  \mathcal R _q(e_j),
\end{align*}
where
\begin{align}\label{e_definitionEllBullet}
  \tlfunc\bullet(e_j)  =  \tilde \sigma_j + \beta_j -2 \pi j \alpha_j ,
\text{ with }
 \tilde {\sigma}_j  = -\int_0^1 \frac{ \mu^2(x)}{6}  e^{2 \pi i j x} dx
\end{align}
and
\begin{align*}
 \mathcal R _q(e_j)=  \frac{1}{q^2}  \sum_{\substack{s\in\Z\setminus\{0\}\\ sq\ne j}} q^2 \Sfo{sq-j}{q} +  \beta_{sq-j} + 2\pi ij  \alpha_{sq-j} + \eps O\left(\frac{j^2}{q^4} \right).
\end{align*}
\end{lemma}
\begin{proof}
First, we claim that:
\begin{align}\label{e_firstStep}
  \tlfunc q(e_j) &= \left(1+\Sfo{0}{q}+\frac{\beta_0}{q^2}\right)\dec qj + \frac{q^2\Sfo{j}{q}+\beta_j-2\pi ij\alpha_j}{q^2}\\
                 &\phantom = +\frac1{q^2} \sum_{\substack{s\in\Z\setminus\{0\}\\ sq\ne j}}({q^2\Sfo{sq-j}{q} +\beta_{sq-j}+2\pi ij\alpha_{sq-j}})+
\eps\cO\left(\frac{ j^2}{q^{4}}\right).\notag
\end{align}
The expression~\eqref{e_firstStep} is the result of the following simple computation: by plugging in $u = e_j$
in~\eqref{e_tlfuncExpression}, we obtain:
  \begin{align*}
    \tlfunc q(e_j) &= \frac1q\sum_{k = 0}^{q-1}\cos\left(2\pi j\left(k/q +
                            \frac{\alpha(k/q)}{q^2}\right)\right)\left(1+\frac{\beta(k/q)}{q^2}\right)+\Sop_q(e_j)+
 O\left( \frac{\eps}{q^{4}}\right)\\
                   &= \frac1q\sum_{k = 0}^{q-1}\left(\cos(2\pi jk/q)-2\pi j\sin(2\pi
                            jk/q)\frac{\alpha(k/q)}{q^2}\right)\left(1+\frac{\beta(k/q)}{q^2}\right)\\
                   &\phantom = +\Sop_q(e_j)+
 O\left( \frac{\eps j^2}{q^{4}}\right)\\
                   &= \dec qj  \\
                   &\phantom = +  \frac1{2q^3}\sum_{k = 0}^{q-1}
\sum_{p\in\Z}\Bigg[2\pi ij(\exp(2\pi ijk/q)-\exp(-2\pi
                     ijk/q))\alpha_p\exp(2\pi ipk/q)+\\
                   &\phantom{= \dec qj +  \frac1{2q^3}\sum_{k = 0}^{q-1}\sum_{p\in\Z}\Bigg[}+(\exp(2\pi ijk/q)+\exp(-2\pi
                     ijk/q))\beta_p\exp(2\pi ipk/q)\Bigg]\\
                   &\phantom = +\Sop_q(e_j)+
O\left( \frac{\eps j^2}{q^{4}}\right)& \\
                   &= \dec qj + \frac1{2q^2}\sum_{s\in\Z}\left[\beta_{sq-j}+2\pi ij\alpha_{sq-j}+\beta_{sq+j}-2\pi
                     ij\alpha_{sq+j}\right]+\Sop_q(e_j)+
 O\left( \frac{\eps j^2}{q^{4}}\right)
                     \intertext{and using the fact that $\alpha_p = -\alpha_{-p}$,
                     $\beta_p = \beta_{-p}$ and by~\eqref{e_definitionSigmaFunctional}}
                   &= \dec qj + \frac1{q^2}\sum_{s\in\Z}\left[q^2 \Sfo{sq-j}{q}+\beta_{sq-j}+2\pi ij\alpha_{sq-j}\right]+\cO\left(\frac{\eps j^2}{q^{4}}\right)\\
  \end{align*}
  which immediately implies~\eqref{e_firstStep}, In order to conclude the proof, we need to control the term $\Sfo jq$
  appearing on the second term in~\eqref{e_firstStep} in a way that is independent of $q$.  To perform this task
    we note that
    \begin{align*}
      \Sfo jq = \int_0^1 S_q(x) e^{2 \pi i j x} dx &
                                                     = \int_0^1 \left ( \frac{ \sin \left ( \frac{1}{q} \mu \left (x \right  ) \right )}{ \frac{1}{q} \mu \left ( x \right  )} -1 \right )  e^{2 \pi i j x} dx  \\
                                                   & = - \int_0^1 \frac{ \mu^2(x)}{6q^2} e^{2 \pi i j x} dx
                                                     + \int_0^1 R \left ( \frac{\mu(x)}{q} \right ) e^{2
                                                     \pi i j x}dx,
\end{align*}
where $R$ is defined by $ \frac{\sin(y)}{y} - 1 = -\frac{y^2}{6} +R(y).$ Since $R(y)=O(y^4)$ and $R'(y)=O(y^3)$, by
performing integration by parts once to the the second integral and the fact $| \mu'(x) | =\cO(\eps)$, we get
$$ \int_0^1 R \left ( \frac{\mu(x)}{q} \right )  e^{2 \pi i j x} dx =
O\left( \frac{\eps j}{q^4}\right). $$ Therefore, we can absorb this term in the remainder term $\mathcal R_q(e_j)$.  We
conclude that we can write
\begin{align*}
 \tilde {\ell}_q(e_j) = \left(1+ \Sfo{0}{q} + \frac{\beta_0}{q^2}\right) \delta_{q|j} + \frac{\tlfunc\bullet(e_j)}{q^2} +  \mathcal R _q(e_j),
\end{align*}

with
$\tlfunc\bullet(e_j) $ and $\tilde {\sigma}_j$
as in the statement of the lemma.
\end{proof}

{\it Acknowledgments.}
We thank H. Eliasson, B. Fayad, A. Figalli, J. Mather, I. Polterovich, P. Sarnak, S. Zeldich for
their most useful comments, which allowed us to vastly improve the exposition of our result.  JDS acknowledges partial
NSERC support. VK acknowledges partial support of the NSF grant DMS-1402164. QW acknowledges support of University of Maryland during her stay in College park where part of the work was done.

  \bibliographystyle{abbrv} \bibliography{dsr}

\end{document}